\documentclass[12pt,a4paper]{article}
\usepackage{amssymb,amsmath,amsthm,epsfig,verbatim,pstricks,url} 
\usepackage{tikz,pgfplots,tikz-3dplot,makecell}
\usepackage{graphbox} 
\usetikzlibrary{fit}
\newcommand\addvmargin[1]{
\node[fit=(current bounding box),inner ysep=#1,inner xsep=0]{};
}
\usepackage{ifpdf}  
\newtheorem{thm}{\sc Theorem.}[section]
\newtheorem{lem}[thm]{\sc Lemma.}
\newtheorem{rem}[thm]{\sc Remark.}
\newtheorem{conjecture}[thm]{\sc Conjecture.}

\newenvironment{AMS}%
{{\upshape\bfseries AMS subject classifications. }\ignorespaces}{}
\newenvironment{keywords}{{\upshape\bfseries Key words. }\ignorespaces}{}

\newcommand{\bRplus}{{\mathbb R}_{>0}}
\newcommand{\bRgeq}{{\mathbb R}_{\geq 0}}
\newcommand{\RZ}{{\mathbb R} \slash {\mathbb Z}}
\newcommand{\RpisZ}{{\mathbb R} \slash (2\,\pi\,\mathfrak s\,{\mathbb Z})}
\newcommand{\RpiZ}{{\mathbb R} \slash (2\,\pi\,{\mathbb Z})}
\newcommand{\bR}{{\mathbb R}}

\newcommand{\bZ}{\mathbb{Z}}
\newcommand{\bS}{{\mathbb S}}
\newcommand{\bGamma}{{\Gamma}}

\newcommand{\bD}{{\mathbb D}}
\newcommand{\bH}{{\mathbb H}}
\newcommand{\distg}{\operatorname{dist}_g}
\newcommand{\spa}{\operatorname{span}}

\newcommand{\dH}[1]{\;{\rm d}{\mathcal{H}}^{#1}} 
\newcommand{\ds}{\;{\rm d}s}
\newcommand{\drho}{\;{\rm d}\rho}
\newcommand{\Vh}{\underline{V}^h}

\newcommand{\Vpartialzero}{\underline{V}_{\partial_0}}

\newcommand{\xspace}{\mathbb {X}}
\newcommand{\yspace}{\mathbb {Y}}
\newcommand{\Vhpartialzero}{\underline{V}^h_{\partial_0}}
\newcommand{\Id}{\rm Id}
\newcommand{\id}{\rm id}

\newcommand{\dd}[1]{\frac{\rm d}{{\rm d}#1}}
\newcommand{\ddt}{\dd{t}}
\newcommand{\ek}{e}
\newcommand{\ttau}{\Delta t}
\newcommand{\BGNmckappa}{\mathcal{A}}

\newcommand{\BGNmc}{\mathcal{C}}
\newcommand{\BGNpwf}{\mathcal{P}}
\newcommand{\BGNpwfwf}{\mathcal{Q}}
\def\epsilon{\varepsilon} 
\newcommand{\mat}[1]{\underline{\underline{#1}}\rule{0pt}{0pt}}

\def\hat{\widehat}

\allowdisplaybreaks
\def\arraystretch{1.15}

\begin{document}
\title{
Numerical approximation of boundary value problems for curvature flow
and elastic flow \\ in Riemannian manifolds}
\author{Harald Garcke\footnotemark[2]\ \and 
        Robert N\"urnberg\footnotemark[3]}

\renewcommand{\thefootnote}{\fnsymbol{footnote}}
\footnotetext[2]{Fakult{\"a}t f{\"u}r Mathematik, Universit{\"a}t Regensburg, 
93040 Regensburg, Germany}
\footnotetext[3]{Department of Mathematics, University of Trento, Trento,
Italy}

\date{}

\maketitle

\begin{abstract}
We present variational
approximations of boundary value problems for 
curvature flow (curve shortening flow)
and elastic flow (curve straightening flow)
in two-dimensional Riemannian manifolds that are conformally flat.
For the evolving open curves we propose natural boundary conditions that
respect the appropriate gradient flow structure. Based on suitable weak
formulations we introduce finite element approximations using piecewise
linear elements. For some of the schemes a stability result can be shown.
The derived schemes can be employed in very different contexts.
For example, we apply the schemes to the Angenent metric 
in order to numerically compute rotationally symmetric self-shrinkers 
for the mean curvature flow.
Furthermore, we utilise the schemes to compute geodesics that 
are relevant for optimal interface profiles in multi-component phase field 
models.
\end{abstract} 



\renewcommand{\thefootnote}{\arabic{footnote}}

\setcounter{equation}{0}
\section{Introduction} 

In this paper we consider numerical approximations for gradient flows of
curves evolving in Riemannian manifolds that are conformally flat. Here we
allow both closed and open curves, where in the latter case appropriate
boundary conditions need be considered in order to respect the required
gradient flow structure.

We define the Riemannian manifold with the help of its metric tensor as
follows.
On a domain $H\subset\mathbb{R}^2$ we let the metric tensor be given by
\begin{equation} \label{eq:g}
[(\vec v, \vec w)_g](\vec z) = g(\vec z)\,\vec v\,.\,\vec w \quad
\forall\ \vec v, \vec w \in \bR^2
\qquad \text{ for } \vec z \in H\,,
\end{equation}
where $\vec v\,.\,\vec w = \vec v^T\,\vec w$ is the standard Euclidean inner
product, and where $g:H \to \bRplus$ is a smooth positive weight 
function. 
This is the setting one obtains for a two-dimensional
Riemannian manifold that is conformally equivalent to the Euclidean
plane. Of course, if $g$ is constant we recover the case of a
Euclidean ambient space.
In local coordinates the metric is precisely given by
\eqref{eq:g}, see e.g.\ \cite{Jost05,Schippers07,Kuhnel15}. 
Examples of such
situations are the hyperbolic plane, the hyperbolic disc and the
elliptic plane. Other examples are given by two-dimensional
manifolds in $\mathbb{R}^d$, $d\ge 3$, that can be conformally
parameterised, such as spheres without pole(s), catenoids and tori.
Coordinates $(x_1,x_2)\in H$ together with a metric
$g$ as in \eqref{eq:g} are called isothermal coordinates, i.e.\ in all
situations considered in this paper we assume that we have isothermal
coordinates. We refer to Section~\ref{sec:old1} and 
\cite[3.29 in Section~3D]{Kuhnel15} for more information. 

The metric tensor \eqref{eq:g} induces a notion of length in $H$.
In particular, the length of a vector $\vec v \in \bR^2$ 
at the location $\vec z \in H$ is defined by
\begin{equation} \label{eq:normg}
[|\vec v|_g](\vec z) = \left([(\vec v, \vec v)_g](\vec z) \right)^\frac12
= g^\frac12(\vec z)\,|\vec v|\,,
\end{equation}
whereas the length of a curve $\vec\gamma \in C^1([0,1], H)$ is given by
\begin{equation} \label{eq:Lgamma}
L_g(\vec\gamma) = \int_0^1 [|\vec\gamma'(\rho)|_g](\vec\gamma(\rho)) \drho
= \int_0^1 g^\frac12(\vec\gamma(\rho))\,|\vec\gamma'(\rho)|\drho\,.
\end{equation}
We note that $L_g$, which is also called geodesic length,
naturally induces a distance function in $H$, with the
distance between two points $\vec z_0, \vec z_1 \in H$ defined as
\begin{equation*} 
\distg(\vec z_0, \vec z_1) = \inf \left\{
L_g(\vec\gamma) : 
\vec\gamma \in C^1([0,1], H)\,,\ \vec\gamma(0) = \vec z_0\,,\ \vec\gamma(1)
= \vec z_1 \right\}.
\end{equation*}
It can be shown that $(H,\distg)$ is a metric space, see
\cite[Section~1.4]{Jost05}.

We will present the mathematical details of curvature flow and elastic flow
in the next section, together with the derivation of suitable boundary
conditions. For now we mention that curvature flow,
for a family of curves $(\Gamma(t))_{t\in[0,T]}$, can be defined as the
$L^2$--gradient flow of geodesic length, 
$L_g(\Gamma) = \int_\Gamma g^\frac12 \ds$,
with respect to the $L^2$--inner product $\langle v,w\rangle = 
\int_\Gamma v\,w\,g^\frac12 \ds$. 
It is often called curve shortening flow.
In particular, on letting the geodesic
curvature $\varkappa_g$ be the first variation of $L_g(\Gamma)$,
with respect to $\langle \cdot,\cdot\rangle$, then curvature flow is given by
\[
\mathcal{V}_g = \varkappa_g\,,
\]
where $\mathcal{V}_g = g^{\frac12}\,\mathcal{V}$, and
$\mathcal V$ is the Euclidean normal velocity of $\Gamma$ in $\bR^2$.
Moreover, the geodesic elastic energy of $\Gamma$ is defined by 
$W_g(\Gamma) = \tfrac12\,\int_\Gamma (\varkappa_g)^2\,g^\frac12 \ds$,
and elastic flow arises as the $L^2$--gradient flow of $W_g(\Gamma)$.
The elastic energy is often called bending energy of $\Gamma$, and elastic 
flow is also known by the name curve straightening flow. 
Critical points of $W_g(\Gamma)$ are called elastica, and a lot of interest 
in elastic flow is driven by the fact that stationary solutions to the flow
are elastica. Moreover, elastic flow can be a viable strategy to obtain
unstable geodesics, i.e.\ critical points of the length $L_g(\Gamma)$
that are unstable under curvature flow. In fact, all geodesics satisfy
$\kappa_g=0$, and so they represent global minimisers for the bending energy.

There is a tremendous amount of work in the literature on curvature flow 
and elastic flow in the Euclidean plane, 
both from an analytical and a numerical point of view.
Curvature flow in more complex ambient spaces has been treated
analytically in e.g.\ \cite{EpsteinC87,Cabezas-RivasM07,AndrewsC17},
while numerical approximations have been considered in 
\cite{ChengBMO02,MikulaS06,SpiraK07,MaitreS08,curves3d,hypbol}, 
for the case of closed curves,
and in \cite{BenninghoffG16} for the case of open curves.
Using the flow along the negative gradient of the total squared geodesic
curvature functional to obtain geodesics as long-time limits has been first
suggested in a seminal paper by Langer and Singer, \cite{LangerS84}. Later the
same authors used Morse theory to investigate
stable critical points of the functional, \cite{LangerS87}. 
The curve straightening flow has been used in \cite{Linner98a} to compute
periodic geodesics, and in \cite{Koiso15} a variant of the flow was used to
study the evolution of an elastic wire in a Riemannian manifold.
Short-time existence, and in certain cases long-time existence, for
elastic flow for closed curves in Riemannian manifolds has been studied
analytically in \cite{DallAcquaLLPS18}, for the case that the manifold is 
a sphere, and in \cite{DallAcquaS17preprint,DallAcquaS18,MullerS20}
for the hyperbolic plane. 
We are not aware of existing work on boundary value problems for elastic 
curves in Riemannian manifolds, but note that the case of 
a Euclidean ambient space has been considered in e.g.\
\cite{Helmers13,DallAcquaLP14,DallAcquaP14,Helmers15,DallAcquaLP17,%
DallAcquaLP19,GarckeMP19,GarckeMP20}.
As far as the numerical approximation of elastic flow is concerned,
we remark that the case of a Euclidean ambient space have been treated 
in \cite{DeckelnickD09,pwf,Bartels13a}. In \cite{DeckelnickD09}
error estimates are shown, while \cite{Bartels13a} contains a partial 
convergence result under a regularity assumption on the velocity.
The case of 
a Riemannian manifold has been studied in \cite{curves3d,hypbol,hypbolpwf}.
All these approaches use finite element discretisations and are of
variational structure.

The present authors, together with John W.\ Barrett, have considered the
evolution of closed curves in Riemannian manifolds that are conformally
equivalent to the Euclidean plane in the recent works \cite{hypbol,hypbolpwf}.
Building on these works, in
this paper we derive boundary value problems for curvature flow and
elastic flow in such manifolds, and we believe that for elastic flow 
the obtained formulations are new in the literature.
Using appropriate variations, different boundary value problems are derived in
Section~\ref{subsec:cf} for curvature flow, see \eqref{eq:cfall}, and 
in Section~\ref{subsec:ef} for elastic flow, see
\eqref{eq:efall}, \eqref{eq:SF1} and \eqref{eq:SF2}.
In the case of elastic flow, the obtained conditions 
generalise classical Navier conditions as well as so called
clamped conditions and semi-free type conditions. 
We will also introduce variational formulations, which lead to natural
spatially discrete and fully discrete approximations for the highly
nonlinear problems. In particular, the variational treatment allows 
for a natural discretisation of boundary conditions,
which in the case of elastic flow are highly non-trivial. 
We introduce finite element schemes with
good mesh properties as well as schemes which allow for a stability
result. We also present several numerical results, which include
computations that are the first for boundary value problems for 
elastic flow in Riemannian manifolds.  
We believe that the presented numerical simulations make a strong
case for the usefulness of curvature flow and elastic flow in computing
geodesics in Riemannian manifolds, both in the case of closed curves,
and in the case of curves with boundary conditions.

We end this introduction with the presentation of some example metrics
for \eqref{eq:g}. To this end, we define the half-plane 
\[
\bH^2 = \{ \vec z \in \bR^2 : \vec z \,.\,\vec\ek_1 > 0 \}
\]
with closure $\overline{\bH^2} = 
\{ \vec z \in \bR^2 : \vec z \,.\,\vec\ek_1 \geq 0 \}$.
Metrics that the authors have considered in their recent works on closed
curves, see \cite{hypbol,hypbolpwf}, are
\begin{subequations} \label{eq:gs}
\begin{align} 
g(\vec z) & = (\vec z \,.\,\vec\ek_1)^{-2\,\mu}\,,\ \mu \in \bR\,,
\quad \text{and} \quad H = 
\begin{cases} \bH^2 & \mu \not=0\,,\\
\bR^2 & \mu = 0\,,\\
\end{cases} \label{eq:gmu}\\
g(\vec z) & = \frac4{(1 - \alpha\, |\vec z|^2)^2}\,,\ \alpha \in \bR\,,
\quad \text{and} \quad H = \begin{cases}
\bD_{\alpha} = \{ \vec z \in \bR^2 : |\vec z| < \alpha^{-\frac12} \}
& \alpha > 0\,, \\
\bR^2 & \alpha \leq 0\,,
\end{cases}\label{eq:galpha}\\
g(\vec z) & = \cosh^{-2}(\vec z\,.\,\vec\ek_1) \quad\text{and}\quad
H = \bR^2\,,\label{eq:gMercator}\\
g(\vec z) & = \cosh^2(\vec z\,.\,\vec\ek_1) \quad\text{and}\quad
H = \bR^2\,, \label{eq:gcatenoid}\\
g(\vec z) & = \mathfrak s^2\,
([\mathfrak s^2 + 1]^\frac12 - \cos (\vec z\,.\,\vec\ek_2))^{-2}\,,\ 
\mathfrak s \in \bRplus\,,
\quad\text{and}\quad H = \bR^2\,. \label{eq:gtorus}
\end{align}
\end{subequations}
Recall that \eqref{eq:gmu} with $\mu=1$ models the hyperbolic plane,
while $\mu=0$ corresponds to the Euclidean plane.
The metric \eqref{eq:galpha} for $\alpha=1$ models the hyperbolic disk,
while $\alpha=-1$ yields the elliptic plane. Moreover,
\eqref{eq:gMercator}, \eqref{eq:gcatenoid} and \eqref{eq:gtorus} arise
from conformal parameterisations of spheres, catenoids and tori,
respectively, where in the latter case the torus has
large radius $[1 + \mathfrak s^2]^\frac12$ and small radius 1.

Additional metrics that we consider in this paper are
\begin{subequations} \label{eq:gnew}
\begin{align} 
  g(\vec z) & = (\vec z\,.\, \vec\ek_1)^{2\,(n-1)} 
\,e^{-\frac12\, |\vec z|^2} \,,\ n \in \bZ_{\geq 2}\,,
\quad\text{and}\quad H = \bH^2\,, \label{eq:gAngenent} \\
g(\vec z)&= \frac{\mathfrak{b}^2}{1 - \mathfrak{b}^2}\,e^{2\,\mathfrak{b}\,\vec
z\,.\,\vec\ek_1}\,,\ \mathfrak{b} \in (0,1)\,,
\quad\text{and}\quad H = \bR^2\,, \label{eq:gcone} \\
g(\vec z)&= \Psi({\bf u}_0 + U \,\vec z)\,,\ {\bf u}_0 \in \bR^3\,,\
U \in \bR^{3\times 2}\,,\ \Psi \in C^\infty(\bR^3)\,,
\quad\text{and}\quad H = \bR^2\,. \label{eq:gGNS}
\end{align}
\end{subequations}
The metric \eqref{eq:gAngenent} is also called the Angenent metric,
see \cite[(5)]{Angenent92}, with a small mistake in the exponent,
and \cite[(1.3)]{KleeneM14}, 
and is of interest in differential geometry. 
Here we mention the fact that for a rotationally symmetric hypersurface 
$\mathcal{S} \subset \bR^{n+1}$, with generating curve $\Gamma \subset \bH^2$,
the geodesic length of $\Gamma$, with respect to the metric
\eqref{eq:gAngenent}, collapses, up to a constant factor, 
to Huisken's $F$-functional
\begin{equation} \label{eq:HuiskenF}
F(\mathcal{S}) 
= (4\,\pi)^{-\frac n2}\,\int_{\mathcal{S}} e^{-\frac14\,|\vec\id|^2}\dH{n}\,,
\end{equation}
see \cite{Huisken90,ColdingM12,Berchenko-Kogan19}. 
It can be shown, \cite{Huisken90}, that
critical points of \eqref{eq:HuiskenF} are self-shrinkers for
mean curvature flow, and so geodesics for the metric \eqref{eq:gAngenent} 
generate axisymmetric self-shrinkers, such
as the Angenent torus, see \cite{Angenent92}.
 
The metric \eqref{eq:gcone}, on the other hand, 
arises from a conformal parameterisation of a right circular cone without the 
apex as follows. Let 
\begin{equation} \label{eq:coneM}
\mathcal{M} = 
\{ (\beta\,r\,\cos\,\theta, \beta\,r\,\sin\,\theta, r)^T : r \in \bRplus\,,
\theta \in [0,2\,\pi)\}\,, \quad \beta \in \bRplus\,.
\end{equation}
We recall from \cite[\S2.4]{hypbol} that  
$\vec\Phi : H \to \mathcal{M}$ is a conformal parameterisation of 
$\mathcal{M}$, if $\mathcal{M} = \vec\Phi(H)$, 
$|\partial_{\vec\ek_1} \vec\Phi(\vec z)|^2 = |\partial_{\vec\ek_2} 
\vec\Phi(\vec z)|^2$ and
$\partial_{\vec\ek_1} \vec\Phi(\vec z) \,.\, 
\partial_{\vec\ek_2} \vec\Phi(\vec z) = 0$ for all $\vec z \in H$.
Using the ansatz
\[
\vec\Phi(\vec z) = r(\vec z\,.\,\vec\ek_1)\,
(\beta\,\cos(\vec z\,.\,\vec\ek_2), \beta\,\sin(\vec z\,.\,\vec\ek_2), 1)^T \,,
\]
for some function $r\in C^\infty(\bR,\bRplus)$,
it is easy to see that $\partial_{\vec\ek_1}\vec\Phi\,.\,
\partial_{\vec\ek_2}\vec\Phi = 0$, as well as
\begin{equation*} 
|\partial_{\vec\ek_1} \vec\Phi|^2 = (1  + \beta^2)\,(r')^2 \quad\text{and}\quad
|\partial_{\vec\ek_2} \vec\Phi|^2 = \beta^2\,r^2 \,,
\end{equation*}
which shows that $r(u) = e^{\mathfrak{b}\,u}$, where
$\mathfrak{b} = [\frac{\beta^2}{1+\beta^2}]^{\frac12}$,
leads to a conformal parameterisation of $\mathcal M$.
In particular, setting $g = |\partial_{\vec\ek_1} \vec\Phi|^2 = 
|\partial_{\vec\ek_2} \vec\Phi|^2$ leads to \eqref{eq:gcone}. 
 
Finally, metrics of the type \eqref{eq:gGNS}, together with the choices
\begin{equation} \label{eq:U}
{\bf u}_0 = (1,0,0)^T\quad\text{and}\quad
U = \begin{pmatrix}
2^{-\frac12} & 6^{-\frac12} \\
-2^{-\frac12} & 6^{-\frac12} \\
0 & -(\frac23)^{\frac12} 
\end{pmatrix}\,,
\end{equation}
play a role in determining optimal interface profiles in multi-component 
Ginzburg--Landau phase field models, see e.g.\ \cite{GarckeNS00}. For example,
the choice 
\begin{equation} \label{eq:Psi}
\Psi(u_1, u_2, u_3) = 
\sigma_{12}\,u_1^2\,u_2^2
+ \sigma_{13}\,u_1^2\,u_3^2 
+ \sigma_{23}\,u_2^2\,u_3^2 
+ \sigma_{123}\,u_1^2\,u_2^2\,u_3^2 \,,
\end{equation}
where $\sigma_{12}\,,\sigma_{13}\,,\sigma_{23}\in \bRplus$ and
$\sigma_{123} \in \bRgeq$,
corresponds to \cite[(24), (25)]{GarckeNS00}, where 
${\bf u} = (u_1,u_2,u_3)^T$
represents a three-phase order parameter and $u_i$ stands for the fraction
of phase $i$. We recall that physically meaningful values for the order
parameter have to lie within the Gibbs simplex
\begin{equation} \label{eq:G}
G = \{ (u_1,u_2,u_3)^T \in \bR^3 : u_1+u_2+u_3 = 1\,, u_1,u_2,u_3 \geq 0\}\,.
\end{equation}
In order to rigorously relate phase field parameters to their
sharp interface limits, it is necessary to establish if the only geodesics 
connecting the three pure phases, ${\bf e}_1= (1,0,0)^T$, ${\bf e}_2=(0,1,0)^T$ 
and ${\bf e}_3=(0,0,1)^T$, are given by straight line segments.
Of course, generalisations 
to other types of potentials are also possible. We refer to
\cite{GarckeNS00,GarckeH08,BretinM17} for more details.

The remainder of this paper is organised as follows.
In Section~\ref{sec:old1} we present strong and weak formulations
of curvature flow and elastic flow. The semidiscrete continuous-in-time
finite element approximations introduced in Section~\ref{sec:sd} will be based
on these weak formulations. Stability of some of the schemes is also shown in  
Section~\ref{sec:sd}. Fully discrete approximations are presented in
Section~\ref{sec:fd}, together with results on their well-posedness and
stability, where applicable. Finally, in Section~\ref{sec:nr} we present
several numerical simulations for the derived schemes and the various metrics
considered in this paper.

\setcounter{equation}{0}
\section{Mathematical formulations} \label{sec:old1}

We let $\RZ$ be the periodic interval $[0,1]$, and set
\[
I = \RZ\,, \text{ with } \partial I = \emptyset\,,\quad \text{or}\quad
I = (0,1)\,, \text{ with } \partial I = \{0,1\}\,.
\]
Consider a family of curves $(\Gamma(t))_{t\in [0,T]}$, $T > 0$, that can
be either open, $I=(0,1)$, or closed, $I=\RZ$.
Given some smooth parameterisation 
$\vec x: I\times[0,T] \ni (\rho,t) \mapsto \vec x(\rho,t)\in \bR^2$,
with $|\vec x_\rho| > 0$ in $\overline I \times [0,T]$,
we introduce the arclength $s$ of the curve, i.e.\ $\partial_s =
|\vec{x}_\rho|^{-1}\,\partial_\rho$, and set
\begin{equation} \label{eq:tau}
\vec\tau = \vec x_s = 
\frac{\vec x_\rho}{|\vec x_\rho|} \quad \mbox{and}
\quad \vec\nu = -\vec\tau^\perp \qquad \text{in } I\,,
\end{equation}
where $\cdot^\perp$ denotes a clockwise rotation by $\frac{\pi}{2}$.
We let $\mathcal V = \vec x_t\,.\,\vec\nu$ denote the normal velocity, and let
the Euclidean curvature $\varkappa$ be defined by
\begin{equation} \label{eq:varkappa}
\varkappa\,\vec\nu 
= \vec x_{ss} = 
\frac1{|\vec x_\rho|} \left[ \frac{\vec x_\rho}{|\vec x_\rho|} \right]_\rho
\qquad \text{in } I\,,
\end{equation}
see \cite{DeckelnickDE05}. We also let
\begin{equation} \label{eq:sg}
\partial_{s_g} = |\partial_\rho\,\vec x|_g^{-1} \,\partial_\rho =
g^{-\frac12}(\vec x)\,|\vec x_\rho|^{-1} \,\partial_\rho
= g^{-\frac12}(\vec x)\,\partial_s \qquad \text{in } I\,.
\end{equation}
We introduce 
\begin{equation} \label{eq:nug}
\vec\nu_g = g^{-\frac12}(\vec x) \,\vec\nu
= - g^{-\frac12}(\vec x) \,\vec x_s^\perp = - \vec x_{s_g}^\perp
\quad \text{and} \quad
\vec\tau_g = \vec x_{s_g} \qquad \text{in } I\,,
\end{equation}
so that $\vec\tau_g\,.\,\vec\nu_g =  0$ and
$|\vec\tau_g|_g^2 = |\vec\nu_g|_g^2 = (\vec\nu_g, \vec\nu_g)_g = g(\vec
x)\,\vec\nu_g\,.\,\vec\nu_g=1$, and let
\begin{equation} \label{eq:Vg}
\mathcal{V}_g = (\vec x_t, \vec\nu_g)_g = g^\frac12(\vec x)\,\vec
x_t\,.\,\vec\nu = g^\frac12(\vec x)\,\mathcal{V} \qquad \text{in } I\,.
\end{equation}
At this stage we would like to draw the reader's
attention to the different usages of subscripts in this paper. The subscripts
$\cdot_g$ above, and throughout the paper, denote quantities
associated with the metric $g$. The subscripts $\cdot_t$ and $\cdot_\rho$, 
on the other hand, denote partial derivatives with respect to $t$ and $\rho$,
respectively. Finally, $\cdot_s$ and $\cdot_{s_g}$ denote weighted 
partial derivatives, and are defined in \eqref{eq:tau} 
and in \eqref{eq:sg}, respectively.

The geodesic curvature can be defined as
\begin{equation} \label{eq:varkappag}
\varkappa_g 
= g^{-\frac12}(\vec x)\left[\varkappa - 
\tfrac12\,\vec\nu\,.\,\nabla\,\ln g(\vec x) \right]
= g^{-\frac12}(\vec x) \left[\varkappa - 
\vec\nu_g\,.\,\nabla\,g^\frac12(\vec x) \right] 
 \qquad \text{in } I\,,
\end{equation}
see \cite{hypbol}.
We note that $\mathcal{V}_g$ and $\varkappa_g$,
similarly to $\mathcal{V}$ and $\varkappa$, only depend on $\Gamma$,
but not on the chosen parameterisation $\vec x$.
For the case of an evolving closed curve, $\partial I = \emptyset$, we
recall from \cite{hypbol} that curvature flow,
\begin{equation} \label{eq:Vgkg}
\mathcal{V}_g = \varkappa_g \qquad \text{in } I\,,
\end{equation}
is the $L^2$--gradient flow of the geodesic length,
\begin{equation} \label{eq:Lg}
L_g(\vec x) = \int_I [|\vec x_\rho|_g](\vec x) \drho
= \int_I g^\frac12(\vec x)\,|\vec x_\rho|\drho\,,
\end{equation}
recall \eqref{eq:Lgamma}. 
In particular, for closed curves evolving by \eqref{eq:Vgkg} it holds that
\begin{equation} \label{eq:gL2gradflow}
\ddt\, L_g(\vec x(t)) + \int_I (\mathcal V_g)^2 \,|\vec x_\rho|_g \drho = 0\,.
\end{equation}

Elastic flow, on the other hand, is the $L^2$--gradient flow of the elastic 
energy
\begin{equation} \label{eq:Wg}
W_g(\vec x) = \tfrac12\,\int_I \varkappa_g^2 \,|\vec x_\rho|_g \drho \,.
\end{equation}
It was shown in \cite{LangerS84}, see also \cite{hypbol},
that for closed curves this flow is given by 
\begin{equation} \label{eq:g_elastflow}
\mathcal{V}_g = - (\varkappa_g)_{s_gs_g} - \tfrac12\,\varkappa_g^3 
- S_0(\vec x)\,\varkappa_g \qquad \text{in } I\,,
\end{equation}
where the Gaussian curvature $S_0$ is defined by
\begin{equation*} 
S_0(\vec z) = - \frac{\Delta\,\ln g(\vec z)}{2\,g(\vec z)}
\qquad \vec z \in H\,,
\end{equation*}
see, e.g., \cite[Definition~2.4]{KrausR13}.
In particular, closed curves evolving by \eqref{eq:g_elastflow} satisfy
\begin{equation} \label{eq:eL2gradflow}
\ddt\,W_g(\vec x(t)) + \int_I (\mathcal{V}_g)^2 \,|\vec x_\rho|_g \drho = 0
\,.
\end{equation}
We state the value of the Gaussian curvature $S_0$ for the metrics
\eqref{eq:gs} and \eqref{eq:gnew} in Table~\ref{tab:S0}. Here we note
that for the Euclidean case, \eqref{eq:gmu} with $\mu=0$, the geodesic
elastic flow \eqref{eq:g_elastflow} collapses to classical elastic flow
$\mathcal V = -\varkappa_{ss} - \tfrac12\,\varkappa^3$.
\begin{table}
\center
\def\arraystretch{1.25}
\begin{tabular}{|c|c|}
\hline
$g$ & $S_0(\vec x)$ \\ \hline
(\ref{eq:gmu}) & 
 $-\mu\,(\vec x\,.\,\vec\ek_1)^{2\,(\mu-1)}$ 
\\
(\ref{eq:galpha}) & $-\alpha$ \\ 
(\ref{eq:gMercator}) & $1$ \\
(\ref{eq:gcatenoid}) & $- \cosh^{-4}(\vec x\,.\,\vec\ek_1)$ \\
(\ref{eq:gtorus}) & 
$\frac{[\mathfrak s^2 + 1]^\frac12 \cos(\vec x\,.\,\vec\ek_2)  - 1}
{\mathfrak s^2}$ \\
(\ref{eq:gAngenent}) & 
$(\vec x\,.\,\vec\ek_1)^{-2\,n}\left[ n - 1
+ (\vec x\,.\,\vec\ek_1)^{2} \right] e^{\frac12\, |\vec x|^2}$ \\ 
\eqref{eq:gcone} & 
0 \\ \hline
\end{tabular}
\caption{The Gaussian curvature $S_0 = - \frac{\Delta\,\ln g}{2\,g}$ 
for the metrics in \eqref{eq:gs} and \eqref{eq:gAngenent}, \eqref{eq:gcone}.}
\label{tab:S0}
\end{table}%

In the remainder of this section, we would like to derive suitable boundary
conditions for curvature flow and elastic flow that respect the gradient flow
structures \eqref{eq:gL2gradflow} and \eqref{eq:eL2gradflow}, and then
introduce weak formulations for the obtained boundary value problems. 
In general, in the case of an open curve, $I=(0,1)$, we would like to 
consider the following types of boundary conditions on $\partial I$:
\begin{equation} \label{eq:newbc}
{\rm(i)}\ \vec x_t = \vec 0\,,\quad
{\rm(ii)}\ \vec x_t \,.\,\vec\ek_1 = 0\,,\quad
{\rm(iii)}\ \vec x_t \,.\,\vec\ek_2 = 0\,.
\end{equation}
Clearly, \eqref{eq:newbc}(i) means that we consider the endpoint fixed in 
time, while in \eqref{eq:newbc}(ii) and \eqref{eq:newbc}(iii) we allow the
boundary point to move freely parallel to the $x_2$-- and $x_1$--axis,
respectively.
For some of the metrics in \eqref{eq:gs} and \eqref{eq:gnew} it is possible to
$C^1$--continuously extend the metric $g$ to $\overline{\bH^2}$ such that 
$g=0$ on the $x_2$--axis. In fact, this holds precisely
for \eqref{eq:gmu} with $\mu \leq -1$ and for \eqref{eq:gAngenent}. 
Having boundary points move freely on the $x_2$--axis in that
case is of particular interest, most notably when the evolving curve plays the
role of the generating curve for an axisymmetric surface.
Altogether, and for later use, we consider the disjoint partition
$\partial I = 
\partial_0 I \cup \partial_{1} I \cup \partial_{2} I\cup \partial_C I \cup
\partial_D I \cup \partial_N I$ with the conditions
\begin{subequations} \label{eq:xp}
\begin{align} 
\vec x_t \,.\,\vec\ek_1 = 0 \quad &
\text{on } \partial_0 I \times (0,T]\,,\label{eq:axibc} \\
\vec x_t = \vec 0 \quad &
\text{on } (\partial_D I \cup \partial_C I \cup \partial_N I) \times (0,T]\,, 
\label{eq:noslipbc} \\
\vec x_t \,.\,\vec\ek_i = 0 \quad &
\text{on } \partial_i I \times (0,T]\,, \ i =1,2\,.  \label{eq:freeslipbc} 
\end{align}
\end{subequations}
In the above $\partial_0 I$ denotes the subset of boundary points
of $I$ that correspond to endpoints of $\Gamma(t)$ where $g$ is set to vanish,
and only in that case does it make sense to consider
\eqref{eq:axibc} separately from \eqref{eq:freeslipbc}.
I.e.\ from now on we will assume
that $g(\vec x(0)) = 0$ on $\partial_0 I$, so that \eqref{eq:axibc} implies
$g(\vec x(t)) = 0$ on $\partial_0 I$ for all $t$.
In our paper, we will consider \eqref{eq:axibc} only for
\eqref{eq:gmu}, with $\mu\leq-1$, and \eqref{eq:gAngenent}. 
The subscripts $D,C,N$ relate to Dirichlet, clamped and Navier
boundary conditions, respectively, with the former relevant for curvature flow,
and the latter two having applications for elastic flow.
In Table~\ref{tab:diagram} we visualise the different types of 
boundary nodes that we consider in this paper.
\begin{table}
\center
\begin{tabular}{ccccc}
\hline
$\partial I$ & $\partial_0 I$ & 
$\partial_D I \cup \partial_C I \cup \partial_N I$ &
$\partial_1 I$ & $\partial_2 I$ \\ \hline
$\partial \Gamma$ &
\begin{tikzpicture}[scale=0.5,baseline=40]
\begin{axis}[axis equal,axis line style=thick,axis lines=center, 
xtick style ={draw=none}, ytick style ={draw=none}, xticklabels = {}, 
yticklabels = {}, xmin=-0.1, xmax = 2, ymin = -2, ymax = 2]
\addplot[mark=*,color=blue,mark size=6pt] coordinates {(0,1)};
\draw[<->,line width=3pt,color=red] (axis cs:0.3,0.5) -- (axis cs:0.3,1.5);
\node at (axis cs:0.5,-0.3){\Large$\vec\ek_1$};
\node at (axis cs:-0.3,0.5){\Large$\vec\ek_2$};
\end{axis}
\addvmargin{1mm}
\end{tikzpicture} 
&
\begin{tikzpicture}[scale=0.5,baseline=40]
\begin{axis}[axis equal,axis line style=thick,axis lines=center, 
xtick style ={draw=none}, ytick style ={draw=none}, xticklabels = {}, 
yticklabels = {}, xmin=-0.1, xmax = 2, ymin = -2, ymax = 2]
\addplot[mark=*,color=blue,mark size=6pt] coordinates {(2,1)};
\node at (axis cs:0.5,-0.3){\Large$\vec\ek_1$};
\node at (axis cs:-0.3,0.5){\Large$\vec\ek_2$};
\end{axis}
\addvmargin{1mm}
\end{tikzpicture} 
&
\begin{tikzpicture}[scale=0.5,baseline=40]
\begin{axis}[axis equal,axis line style=thick,axis lines=center, 
xtick style ={draw=none}, ytick style ={draw=none}, xticklabels = {}, 
yticklabels = {}, xmin=-0.1, xmax = 2, ymin = -2, ymax = 2]
\addplot[mark=*,color=blue,mark size=6pt] coordinates {(2,1)};
\draw[<->,line width=3pt,color=red] (axis cs:2.3,0.5) -- (axis cs:2.3,1.5);
\draw[thick,color=blue] (axis cs:2,-2) -- (axis cs:2,2);
\node at (axis cs:0.5,-0.3){\Large$\vec\ek_1$};
\node at (axis cs:-0.3,0.5){\Large$\vec\ek_2$};
\end{axis}
\addvmargin{1mm}
\end{tikzpicture} 
&
\begin{tikzpicture}[scale=0.5,baseline=40]
\begin{axis}[axis equal,axis line style=thick,axis lines=center, 
xtick style ={draw=none}, ytick style ={draw=none}, xticklabels = {}, 
yticklabels = {}, xmin=-0.1, xmax = 2, ymin = -2, ymax = 2]
\addplot[mark=*,color=blue,mark size=6pt] coordinates {(2,1)};
\draw[<->,line width=3pt,color=red] (axis cs:1.5,0.7) -- (axis cs:2.5,0.7);
\draw[thick,color=blue] (axis cs:-2,1) -- (axis cs:4,1);
\node at (axis cs:0.5,-0.3){\Large$\vec\ek_1$};
\node at (axis cs:-0.3,0.5){\Large$\vec\ek_2$};
\end{axis}
\addvmargin{1mm}
\end{tikzpicture} 
\\ \hline
\end{tabular}
\caption{The different types of boundary nodes enforced by 
(\ref{eq:axibc})--(\ref{eq:freeslipbc}), and their effect on the possible
movement of the boundary points.
}
\label{tab:diagram}
\end{table}%

For some of the weak formulations, it will be useful to have an analogue of
\eqref{eq:varkappa} for the geodesic curvature $\varkappa_g$,
recall \eqref{eq:varkappag}. To this end, we note that it
can be easily shown from \eqref{eq:varkappa} that
\begin{equation} \label{eq:ng}
\nabla\,g^\frac12(\vec x)
= \vec\nu\,(\vec\nu\,.\,\nabla)\,g^\frac12(\vec x)
+ \frac1{|\vec x_\rho|}\left[
g^\frac12(\vec x)\,\frac{\vec x_\rho}{|\vec x_\rho|}\right]_\rho
-g^\frac12(\vec x)\,\varkappa\,\vec\nu
\qquad \text{in } I\,,
\end{equation}
see \cite[(2.16)]{hypbol}.
Combining (\ref{eq:varkappag}) and (\ref{eq:ng}) yields that
\begin{equation} \label{eq:gkgnu}
g(\vec x)\,\varkappa_g\,\vec\nu = 
\frac1{|\vec x_\rho|} \left[ g^\frac12(\vec x)\,
\frac{\vec x_\rho}{|\vec x_\rho|}\right]_\rho - \nabla\,g^\frac12(\vec x)
\qquad \text{in } I\,.
\end{equation}

Let $(\cdot,\cdot)$ denote the $L^2$--inner product on $I$, and let 
\begin{subequations} \label{eq:spaces}
\begin{align} \label{eq:Vpartialzero}
\Vpartialzero &= \{ \vec\eta \in [H^1(I)]^2 : \vec\eta(\rho)\,.\,\vec\ek_1 = 0
\quad \forall\ \rho \in \partial_0 I\}\,, \\
\xspace & = \left\{ \vec\eta \in \Vpartialzero : 
\vec\eta(\rho) = \vec 0\quad \forall\ \rho \in \partial_D I \cup \partial_C I \cup \partial_N I\,,
\ \vec\eta(\rho)\,.\,\vec\ek_i = 0\quad \forall\ \rho \in \partial_i I\,, 
i = 1,2
\right\}\,. \label{eq:xspace}
\end{align}
We also define
\begin{equation} \label{eq:yspace}
\yspace = \left\{ \vec\eta \in \Vpartialzero : 
\vec\eta(\rho) = \vec 0 \quad \forall\ \rho \in \partial_N I \right\}\,.
\end{equation}
\end{subequations}

\subsection{Curvature flow} \label{subsec:cf}

For curvature flow we assume that $\partial I = \partial_0 I \cup
\partial_1 I \cup \partial_2 I\cup \partial_C I$, 
where $\partial_0 I$ will only be nonempty for the metrics
\eqref{eq:gmu}, with $\mu\leq-1$, and \eqref{eq:gAngenent}. 

It holds that
\begin{equation}
\ddt\, L_g(\vec x(t)) = \int_I \left[ 
\nabla\,g^\frac12(\vec x)\,.\,\vec x_t + g^\frac12(\vec x)\,
\frac{(\vec x_t)_\rho\,.\,\vec x_\rho}{|\vec x_\rho|^2}
\right] |\vec x_\rho| \drho\,.
\label{eq:dLdt}
\end{equation}
Combining (\ref{eq:dLdt}), (\ref{eq:ng}), \eqref{eq:nug}, (\ref{eq:Vg}) 
and \eqref{eq:normg} yields that
\begin{align}
\ddt\, L_g(\vec x(t)) & = 
\int_I \left(\nabla\,g^\frac12(\vec x)
- \frac1{|\vec x_\rho|}
 \left[ g^\frac12(\vec x)\,\frac{\vec x_\rho}{|\vec x_\rho|}\right]_\rho 
\right).\,\vec x_t \,|\vec x_\rho| \drho 
- \sum_{p \in\partial I} (-1)^{p}\,
[g^\frac12(\vec x)\,\vec x_t\,.\,\vec\tau](p)
\nonumber \\ & 
= \int_I \left[ \vec\nu\,.\,\nabla\,g^\frac12(\vec x)
-g^\frac12(\vec x)\,\varkappa \right] \vec\nu\,. \,\vec x_t \,|\vec x_\rho| 
 \drho 
- \sum_{p \in\partial I} (-1)^{p}\,[g^\frac12(\vec x)\,\vec x_t\,.\,\vec\tau](p)
\nonumber \\ & 
= \int_I \left[ \vec\nu_g\,.\,\nabla\,g^\frac12(\vec x)
- \varkappa\right] \mathcal{V}_g \,|\vec x_\rho| \drho
- \sum_{p \in\partial I}(-1)^{p}\,[g^\frac12(\vec x)\,\vec x_t\,.\,\vec\tau](p)
 \nonumber \\ & 
= - \int_I g^{-\frac12}(\vec x) \left[\varkappa - 
\vec\nu_g\,.\,\nabla\,g^\frac12(\vec x) \right] \mathcal{V}_g 
\,|\vec x_\rho|_g \drho 
- \sum_{p \in\partial I} (-1)^{p}\,[g^\frac12(\vec x)\,\vec x_t\,.\,\vec\tau](p)
\nonumber \\ &
= - \int_I \varkappa_g\,\mathcal{V}_g \,|\vec x_\rho|_g \drho 
- \sum_{p \in\partial I} (-1)^{p}\,[g^\frac12(\vec x)\,\vec x_t\,.\,\vec\tau](p)
\,,
\label{eq:dLdtV}
\end{align}
where, we have recalled \eqref{eq:varkappag}. 
Clearly, the curvature $\varkappa_g$ is the first variation of the length
(\ref{eq:Lg}).

For the case that $\partial_0 I \not= \emptyset$, we note that in order for the
right hand side of \eqref{eq:dLdtV} to remain bounded, it is appropriate to
require that the term $\vec\nu\,.\,\nabla\,g^\frac12(\vec x)$ in the second 
line remains bounded as we approach $\partial_0 I$.
In view of 
\[
\vec\nu\,.\,\nabla\,g^\frac12(\vec x) = 
\tfrac12\,g^{-\frac12}(\vec x)\,\vec\nu\,.\,\nabla\,g(\vec x)\,,
\]
with $\nabla\,g(\vec x)\,.\,\vec\ek_2 = 0$ on $\partial_0 I$, 
since $g=0$ on $\partial_0 I$, we see that
\begin{equation} \label{eq:nue1}
\vec\nu\,.\,\vec\ek_1 = 0 \qquad \text{on } \partial_0 I
\qquad\iff\qquad
\vec x_\rho \,.\,\vec\ek_2 = 0 \qquad \text{on } \partial_0 I
\end{equation}
is a natural assumption to make. 
Moreover, in situations where the curve $\Gamma(t) = \vec x(\overline I, t)$
models the generating curve of an axisymmetric surface that is rotationally
symmetric with respect to the $x_2$--axis, the condition \eqref{eq:nue1} 
ensures that the modelled surface is smooth; see also
\cite{aximcf,axisd} for more details.

Similarly to the closed curve case, as discussed in \cite{hypbol}, we
note from (\ref{eq:dLdtV}) that \eqref{eq:Vgkg} 
is the natural $L^2$--gradient flow of $L_g$ with respect to the metric induced
by $g$, i.e.\ it satisfies \eqref{eq:gL2gradflow},
if the boundary conditions \eqref{eq:xp} hold, together with
\[
g^\frac12(\vec x)\,\vec x_t \,.\,\vec\tau = 0 \quad\text{on }\ \partial I\,.
\]
This condition holds automatically on $\partial_0 I$ and $\partial_D I$, while
on the remainder of $\partial I$ we require
\begin{equation*} 
\vec x_\rho \,.\,\vec\ek_2 = 0 
\quad \text{on }\ \partial_1 I
\qquad \text{and} \qquad
\vec x_\rho \,.\,\vec\ek_1  = 0 
\quad \text{on }\ \partial_2 I\,.
\end{equation*}
In terms of the Euclidean properties of the curve $\Gamma(t)$, 
geodesic curvature flow, i.e.\ the evolution equation \eqref{eq:Vgkg}, 
can be written as
\begin{equation} \label{eq:Vk}
g(\vec x)\,\vec x_t\,.\,\vec\nu = 
\varkappa - \tfrac12\,\vec\nu\,.\,\nabla\,\ln g(\vec x) \qquad \text{in } I\,.
\end{equation}
This formulation gives another interpretation for the boundary condition
\eqref{eq:nue1}. In fact, imposing \eqref{eq:nue1} is necessary to allow 
the right hand side of \eqref{eq:Vk} to remain bounded as we approach
$\partial_0 I$. In this way, we restrict ourselves to the class of solutions
to the evolution equation \eqref{eq:Vk}, where the normal velocity and
curvature remain bounded.
Hence in this paper, geodesic curvature flow for an open or closed curve
is given by:
\begin{subequations} \label{eq:cfall}
\begin{alignat}{2}
\mathcal{V}_g & = \varkappa_g && \qquad \text{in } I\,, \label{eq:cfalla} \\
\vec x_t \,.\,\vec\ek_1 & = 0\,,\quad \vec x_\rho \,.\,\vec\ek_2 = 0 && \qquad 
\text{on } \partial_0 I\,, \label{eq:cfallb}\\
\vec x_t & = \vec 0 && \qquad 
\text{on } \partial_D I \,, \\
\vec x_t \,.\,\vec\ek_i & = 0\,, \quad
\vec x_\rho\,.\,\vec\ek_{3-i} = 0 && \qquad 
\text{on } \partial_i I\,, \ i =1,2\,.\label{eq:cfalld}
\end{alignat}
\end{subequations}
We remark that the condition $\vec x_\rho\,.\,\vec\ek_{3-i}=0$ in
\eqref{eq:cfalld} corresponds to a $90^\circ$ contact angle condition,
which is the same as for classical Euclidean curvature flow. In particular,
it does not depend on the chosen metric $g$. That is because in this paper we 
consider conformal metrics, and so compared to the Euclidean case
only the measurement of length changes,
but the measurement of angles remains the same.

A weak formulation of curvature flow, \eqref{eq:cfall}, 
based on the strong formulation
\eqref{eq:Vk} in place of \eqref{eq:cfalla}, is given as follows,
where we also recall \eqref{eq:varkappa} and \eqref{eq:spaces}. 

\noindent
$(\BGNmckappa)$:
Let $\vec x(0) \in \Vpartialzero$. For $t \in (0,T]$
find $\vec x(t) \in [H^1(I)]^2$, with $\vec x_t(t) \in \xspace$, 
and $\varkappa(t)\in L^2(I)$ such that
\begin{subequations} \label{eq:cfweakA}
\begin{align} \label{eq:cfweakAa}
& \left( g(\vec x)\,\vec x_t\,.\,\vec\nu,\chi\,|\vec x_\rho| \right)
= \left( \varkappa - \tfrac12\,\vec\nu\,.\,\nabla\,\ln g(\vec x) 
, \chi\,|\vec x_\rho| \right) \quad \forall\ \chi \in L^2(I)\,,
\\
& \left( \varkappa\,\vec\nu,\vec\eta\, |\vec x_\rho| \right)
+ \left( \vec x_\rho,\vec\eta_\rho\,|\vec x_\rho|^{-1} \right) = 0
\quad \forall\ \vec\eta \in \xspace\,.
\label{eq:cfweakAb}
\end{align}
\end{subequations}
We note that the boundary conditions for $\vec x_t$ in
\eqref{eq:cfall} are enforced through the trial space $\xspace$, recall
\eqref{eq:spaces}, 
while the boundary conditions on $\vec x_\rho$ in \eqref{eq:cfall} follow from
\eqref{eq:cfweakAb}.

An alternative weak formulation, based directly on \eqref{eq:cfall},
together with \eqref{eq:gkgnu}, is given as follows.

\noindent
$(\BGNmc)$:
Let $\vec x(0) \in \Vpartialzero$. For $t \in (0,T]$
find $\vec x(t) \in [H^1(I)]^2$, with $\vec x_t(t) \in \xspace$, 
and $\varkappa_g(t) \in L^2(I)$ such that
\begin{subequations} \label{eq:cfweakC}
\begin{align}
& \left( g(\vec x)\,\vec x_t\,.\,\vec\nu,
\chi\,|\vec x_\rho|\right)
= \left( g^\frac12(\vec x)\,\varkappa_g,\chi\,
|\vec x_\rho|\right) \quad \forall\ \chi \in L^2(I)\,, \label{eq:cfweakCa} 
\\
& \left( g(\vec x)\,\varkappa_g\,\vec\nu,
\vec\eta\,|\vec x_\rho|\right)
+ \left( \nabla\,g^\frac12(\vec x)\,.\,\vec\eta
+ g^\frac12(\vec x)\,\frac{\vec x_\rho\,.\,\vec\eta_\rho}{|\vec x_\rho|^2}
, |\vec x_\rho| \right) 
= 0 \quad \forall\ \vec\eta \in \xspace\,.
\label{eq:cfweakCb} 
\end{align}
\end{subequations}
Once again, the boundary conditions for $\vec x_t$ in
\eqref{eq:cfall} are enforced through the trial space $\xspace$, 
while the conditions on $\vec x_\rho$ in \eqref{eq:cfalld} follow 
directly from \eqref{eq:cfweakCb}. 
In addition, 
it can be shown that \eqref{eq:cfweakCb}, for the metrics \eqref{eq:gmu}, 
with $\mu\leq-1$, and \eqref{eq:gAngenent}, 
also enforces the condition on $\vec x_\rho$ in \eqref{eq:cfallb}.
In fact, this follows by using the techniques in \cite[Appendix~A]{aximcf},
and noting that for both metrics it holds
that $\nabla\,g^\frac12(\vec x)\,.\,\vec\ek_2 = 0$ on $\partial_0 I$,
see Table~\ref{tab:g} below.

\subsection{Elastic flow} \label{subsec:ef}
For elastic flow we assume that
$\partial I = \partial_0 I \cup \partial_C I \cup \partial_N I$, 
where, as before, $\partial_0 I$ will only be nonempty for the metrics
\eqref{eq:gmu}, with $\mu\leq-1$, and \eqref{eq:gAngenent}. 

In order to discuss appropriate boundary conditions consistent with the
gradient flow structure \eqref{eq:eL2gradflow}, we need to re-visit the 
derivation of
\[
\mathcal{V}_g = - (\varkappa_g)_{s_gs_g} - \tfrac12\,\varkappa_g^3 
- S_0(\vec x)\,\varkappa_g \qquad \text{in } I\,,
\]
recall \eqref{eq:g_elastflow}, as presented in \cite[\S2.3]{hypbol}. 
Summarising the authors'
procedure there, they inferred by careful calculation that
\begin{equation} \label{eq:hypbol}
\ddt\,W_g(\vec x(t)) 
= \int_I \left[ (\varkappa_g)_{s_gs_g} + \tfrac12\,\varkappa_g^3
+ S_0(\vec x)\,\varkappa_g
\right]\mathcal{V}_g \,|\vec x_\rho|_g \drho 
\end{equation}
for closed curves, cf.\ \cite[(2.55)]{hypbol}, 
which concurs with the result of \cite[p.\ 3]{LangerS84}.
In order to generalise their work to the case of
open curves, we observe that boundary terms on the right hand side of
\eqref{eq:hypbol} would only be created through applications of integration 
by parts. In the derivation in \cite[\S2.3]{hypbol} this occurs only
in the third line of (2.47) and in the second line of (2.52). 
Regarding the former, we note that for open curves we obtain for the term in
question that
\begin{align*} 
\tfrac12\,\int_I g^{\frac12}(\vec x)\, \varkappa_g^2 \,
\vec x_\rho\,.\,(\vec x_t)_\rho \, |\vec x_\rho|^{-1} \drho &
= - \tfrac12\,\int_I (g^{\frac12}(\vec x)\, \varkappa_g^2 \,
\vec x_s)_s\,.\,\vec x_t \, |\vec x_\rho| \drho \nonumber \\ & \qquad
- \tfrac12\,
 \sum_{p\in\partial I} (-1)^p\,[g^{\frac12}(\vec x)\, \varkappa_g^2
  \,\vec x_t\,.\,\vec\tau](p) \,. 
\end{align*}
In addition, the three applications of integration by parts in 
\cite[(2.52)]{hypbol} give rise to the following boundary terms:
\begin{align*}
& \int_I \varkappa_g\,\mathcal{V}_{ss}\, |\vec x_\rho| \drho
+ \tfrac12\,\int_I \varkappa_g\,\vec x_s\,.\,\nabla\,\ln g(\vec x)
\,\mathcal{V}_s \,|\vec x_\rho| \drho \\ & \quad
 = - \int_I (\varkappa_g)_{s} \,\mathcal{V}_s\,|\vec x_\rho| \drho 
- \sum_{p\in\partial I} (-1)^p\, [\varkappa_g\,\mathcal{V}_s](p) 
\nonumber \\ & \qquad
- \tfrac12\,\int_I (\varkappa_g\,\vec x_s\,.\,\nabla\,\ln g(\vec x))_s
\,\mathcal{V} \,|\vec x_\rho| \drho
- \tfrac12\,\sum_{p\in\partial I} (-1)^p\, 
[\varkappa_g\,\vec x_s\,.\,\nabla\,\ln g(\vec x)\,\mathcal{V}](p)
\nonumber \\ & \quad
 = \int_I (\varkappa_g)_{ss} \,\mathcal{V}\,|\vec x_\rho| \drho 
+ \sum_{p\in\partial I} (-1)^p\, [(\varkappa_g)_s\,\mathcal{V}](p) 
- \sum_{p\in\partial I} (-1)^p\, [\varkappa_g\,\mathcal{V}_s](p) 
\nonumber \\ & \qquad
- \tfrac12\,\int_I (\varkappa_g\,\vec x_s\,.\,\nabla\,\ln g(\vec x))_s
\,\mathcal{V} \,|\vec x_\rho| \drho
- \tfrac12\,\sum_{p\in\partial I} (-1)^p\, 
[\varkappa_g\,\vec x_s\,.\,\nabla\,\ln g(\vec x)\,\mathcal{V}](p)\,.
\end{align*}
Hence, overall, we obtain the boundary terms
\begin{equation} \label{eq:B123}
- \sum_{p\in\partial I} (-1)^p\left[ \tfrac12\,
g^{\frac12}(\vec x)\, \varkappa_g^2
\,\vec x_t\,.\,\vec\tau
- [(\varkappa_g)_s - \tfrac12\,\varkappa_g\,(\ln g(\vec x))_s]
\,\mathcal{V} + \varkappa_g\,\mathcal{V}_s\right](p) \,.
\end{equation}
An alternative way to derive the boundary terms uses the approach of
Langer and Singer, see \cite[p.\ 3]{LangerS84}. 
We now derive natural conditions that make \eqref{eq:B123} vanish for the
boundary conditions considered in \eqref{eq:xp}. 
On $\partial_C I \cup \partial_N I$, the first two terms vanish. 
On $\partial_N I$ we require $\varkappa_g=0$ to make the third term zero, 
while the clamped boundary conditions 
\begin{equation} \label{eq:veczeta}
(-1)^{\id+1}\,\vec\tau = \vec\zeta
\quad \text{on }\ \partial_C I\,,
\end{equation}
with $\vec\zeta : \partial_C I \to \bS^1 = \{\vec z : \bR^2 : |\vec z| = 1\}$,
ensure as usual that $\mathcal V_s = 0$ on $\partial_C I$;
see e.g.\ Lemma~37(ii) in \cite{bgnreview}. 
For the compact notation in \eqref{eq:veczeta} we have used the fact that
the identity function $\id$ always maps elements of $\partial I$ to either 0 
or 1.
On $\partial_0 I$ the first term in \eqref{eq:B123} is zero, and on requiring
\begin{equation} \label{eq:kappags}
(\varkappa_g)_s - \tfrac12\,\varkappa_g\,(\ln g(\vec x))_s = 0
\qquad \text{on }\ \partial_0 I
\end{equation}
we can make the second term vanish. The third term
vanishes if $\mathcal V_s = 0$ on $\partial_0 I$, which similarly to the
clamped boundary conditions follows from ensuring the 
natural assumption \eqref{eq:nue1}. 

Overall, we obtain the following strong formulation for
elastic flow for open or closed curves, consistent with the 
gradient flow structure \eqref{eq:eL2gradflow}. 
\begin{subequations} \label{eq:efall}
\begin{alignat}{2}
\mathcal{V}_g & = - (\varkappa_g)_{s_gs_g} - \tfrac12\,\varkappa_g^3 
- S_0(\vec x)\,\varkappa_g && \qquad \text{in } I\,,\\
\vec x_t \,.\,\vec\ek_1 & = 0\,,\quad
\vec x_\rho \,.\,\vec\ek_2 = 0\,,\quad
[\varkappa_g]_\rho - \tfrac12\,\varkappa_g\,[\ln g(\vec x)]_\rho = 0 && 
\qquad \text{on }\ \partial_0 I\,,\label{eq:efbc0} \\
\vec x_t & = \vec 0\,,\quad (-1)^{\id+1}\,\vec\tau = \vec\zeta &&
\qquad \text{on }\ \partial_C I\,,\label{eq:efbcC} \\
\vec x_t & = \vec 0\,,\quad \varkappa_g = 0 &&
\qquad \text{on }\ \partial_N I\,. \label{eq:efbcN} 
\end{alignat}
\end{subequations}
Other types of boundary conditions, corresponding to so-called free
and semi-free boundary nodes, see e.g.\ \cite{axipwf}, are also
possible. For example, in the semi-free cases one could require 
\begin{equation}\label{eq:SF1}
  \vec x_t\,.\,\vec\ek_i = 0\,,\quad \vec x_\rho\,.\,\vec\ek_{3-i} =
  0\,, \quad [\varkappa_g]_\rho -\tfrac12 \varkappa_g[\ln g(\vec
  x)]_\rho = 0 \qquad \text{on }\ \partial_i I\,,\ i=1,2\,,
\end{equation}
or
\begin{equation}\label{eq:SF2}
  \vec x_t\,.\,\vec\ek_i = 0\,,\quad \varkappa_g = 0\,, \quad [\varkappa_g]_\rho -\tfrac12 \varkappa_g[\ln g(\vec x)]_\rho = 0 
\qquad \text{on }\ \partial_i I\,,\ i=1,2\,.
\end{equation}
These conditions will also lead to vanishing boundary terms in \eqref{eq:B123}.
Observe that \eqref{eq:SF1} involves a $90^\circ$ contact angle condition,
while \eqref{eq:SF2} does not fix the angle but requires curvature to be zero,
similarly to the Navier condition \eqref{eq:efbcN}. 
The boundary condition 
$[\varkappa_g]_\rho -\tfrac12 \varkappa_g[\ln g(\vec x)]_\rho = 0$
seems to be completely new in the literature, as well as the various
combinations of it with other boundary conditions.
In order to simplify the presentation of what follows, we will
concentrate on the conditions in \eqref{eq:efall} from now
on. However, using the techniques from \cite{pwf,axipwf} it is
straightforward to extend our weak formulations and finite element
approximations to these other types of boundary conditions as well.

Combining the techniques in \cite{hypbolpwf,axipwf}, it is not difficult to
derive the following two weak formulations of elastic flow. Here the equations
in the interior of $I$ are the same as for $(\BGNpwf)$ and $(\BGNpwfwf)$
in \cite{hypbolpwf}, while the treatment of the boundary conditions,
achieved through the spaces $\xspace$ and $\yspace$ from \eqref{eq:spaces}, 
is very
similar to the approach taken in \cite{axipwf} in the context of axisymmetric
Willmore flow.

\noindent
$(\BGNpwf)$:
Let $\vec x(0) \in \Vpartialzero$. For $t \in (0,T]$
find $\vec x(t) \in [H^1(I)]^2$, with $\vec x_t(t) \in \xspace$, 
$\vec y(t) \in \yspace$ and $\varkappa \in L^2(I)$ such that 
\begin{subequations} \label{eq:efweakP}
\begin{align}
& \left(g^\frac32(\vec x)\,\vec x_t\,.\,\vec\nu, \vec\chi\,.\,\vec\nu\,
|\vec x_\rho| \right)
= -\tfrac12 \left( g^{-\frac12}(\vec x)
\left(\varkappa - \tfrac12\,\vec\nu\,.\,\nabla\,\ln g(\vec x) \right)^2
, \vec\chi_s\,.\,\vec\tau\,|\vec x_\rho| \right) \nonumber \\ & \qquad 
+ \tfrac14 \left( g^{-\frac12}(\vec x)
 \left(\varkappa - \tfrac12\,\vec\nu\,.\,\nabla\,\ln g(\vec x) \right)^2
, \vec\chi\,.\,(\nabla\,\ln\,g(\vec x))\,|\vec x_\rho| \right)
\nonumber \\ & \qquad
+ \tfrac12 \left( 
g^{-\frac12}(\vec x)
\left(\varkappa - \tfrac12\,\vec\nu\,.\,\nabla\,\ln g(\vec x) \right) 
\,\vec\nu ,(D^2\,\ln\,g(\vec x))\,
\vec\chi\,|\vec x_\rho| \right) \nonumber \\ & \qquad
-\tfrac12\left( g^{-\frac12}(\vec x)
\left(\varkappa - \tfrac12\,\vec\nu\,.\,\nabla\,\ln g(\vec x) \right) 
[\ln\,g(\vec x)]_s, \vec\nu\,.\,\vec\chi_s
\,|\vec x_\rho| \right)
+ \left(\vec y_s\,.\,\vec\nu, \vec\chi_s\,.\,\vec\nu\,|\vec x_\rho| \right)
\nonumber \\ & \qquad
+ \left( \varkappa\, \vec y^\perp,
\vec\chi_s\,|\vec x_\rho| \right)
\qquad \forall  \ \vec\chi \in \xspace\,, \label{eq:efweakPa} \\
& \left(g^{-\frac12}(\vec x)
\left(\varkappa - \tfrac12\,\vec\nu\,.\,\nabla\,\ln g(\vec x) \right) 
- \vec y\,.\,\vec\nu, \chi\,|\vec x_\rho| \right) = 0
\qquad \forall\ \chi \in L^2(I)\,, \label{eq:efweakPb} \\
& \left(\varkappa\,\vec\nu,\vec\eta\, |\vec x_\rho| \right)
+ \left( \vec x_s,\vec\eta_s\,|\vec x_\rho| \right) 
= \sum_{p \in \partial_C I} [\vec\zeta\,.\,\vec\eta](p)
\quad \forall\ \vec\eta \in \yspace\,. \label{eq:efweakPc}
\end{align}
\end{subequations}
The consistency of \eqref{eq:efweakP}, in the case $I = \RZ$, was shown in 
\cite[Appendix~A.1]{hypbolpwf}. As far as the boundary conditions are
concerned, we note that the conditions on $\vec x_t$ in \eqref{eq:efall} 
are enforced through the trial space $\xspace$, recall \eqref{eq:spaces}.
The second conditions in \eqref{eq:efbc0} and \eqref{eq:efbcC}, respectively,
are both enforced through \eqref{eq:efweakPc}.
In addition, we note from \eqref{eq:efweakPb} and \eqref{eq:varkappag}
that $\vec y\,.\,\vec\nu = \varkappa_g$, and so 
the trial space $\yspace$ yields the second condition in \eqref{eq:efbcN}. 
It remains to validate that \eqref{eq:efweakPa} weakly enforces 
the third condition in \eqref{eq:efbc0}. This can be achieved
on closely following the argument in \cite[(A.3)--(A.9)]{hypbolpwf},
noting in particular that the integration by parts in \cite[(A.5)]{hypbolpwf} 
gives rise to the boundary term 
$(\varkappa_g)_s - \tfrac12\,\varkappa_g\,(\ln g(\vec x))_s$ on 
$\partial_0 I$, which enforces \eqref{eq:kappags}, as required.

\noindent
$(\BGNpwfwf)$:
Let $\vec x(0) \in \Vpartialzero$. For $t \in (0,T]$
find $\vec x(t) \in [H^1(I)]^2$, with $\vec x_t(t) \in \xspace$, 
$\vec y_g(t) \in \yspace$ and $\varkappa_g \in L^2(I)$ such that 
\begin{subequations} \label{eq:efweakQ}
\begin{align} 
& \left(g(\vec x)\,\vec x_t\,.\,\vec\nu, \vec\chi\,.\,\vec\nu\,
|\vec x_\rho|_g \right) 
\nonumber \\ &\
= -\tfrac12 \left( \varkappa_g^2 
- \vec y_g\,.\,\nabla\,\ln\,g(\vec x), \left[\vec\chi_s\,.\,\vec\tau + 
\tfrac12\,\vec\chi\,.\,\nabla\,\ln\,g(\vec x)\right] |\vec x_\rho|_g
 \right) \nonumber \\ & \quad 
+ \tfrac12 \left((D^2\,\ln\,g(\vec x))\,\vec y_g, \vec\chi\,
|\vec x_\rho|_g \right) 
+  \left( g^\frac12(\vec x)\,\varkappa_g\,\vec y_g\,.\,\vec\nu
+ \tfrac12\,(\vec y_g)_s\,.\,\vec\tau, \vec\chi\,.\,(\nabla\,\ln\,g(\vec x))\,
|\vec x_\rho|_g \right) \nonumber \\ & \quad 
+ \left( g^\frac12\,\varkappa_g,
\vec\chi_s\,.\,\vec y_g^\perp\,|\vec x_\rho|_g \right)
+ \left( (\vec y_g)_s\,.\,\vec\nu, 
\vec\chi_s\,.\,\vec\nu\,|\vec x_\rho|_g  \right)
\qquad \forall\ \vec\chi \in \xspace\,,\label{eq:efweakQa} \\
& \left(\varkappa_g - g^\frac12(\vec x)\,\vec y_g\,.\,\vec\nu, 
\chi\,|\vec x_\rho|_g \right) = 0
\qquad \forall\ \chi \in L^2(I)\,, \label{eq:efweakQb} \\
& \left(g^\frac12(\vec x)\,\varkappa_g\,\vec\nu, \vec\eta\,|\vec x_\rho|_g
\right) 
+ \left(\vec x_s,\vec\eta_s\,|\vec x_\rho|_g\right) 
+ \tfrac12 \left( \nabla\,\ln\,g(\vec x), \vec\eta\,|\vec x_\rho|_g\right) 
= \sum_{p \in \partial_C I} [g^{\frac12}(\vec x)\,\vec\zeta\,.\,\vec\eta](p)
\nonumber \\ & \hspace{11cm}
\quad \forall\ \vec\eta \in \yspace\,. \label{eq:efweakQc}
\end{align}
\end{subequations}
The consistency of \eqref{eq:efweakQ}, in the case $I = \RZ$, was shown in 
\cite[Appendix~A.2]{hypbolpwf}. 
We note that the boundary conditions on $\vec x_t$ in \eqref{eq:efall} 
are once again enforced through the trial space $\xspace$.
The second condition in \eqref{eq:efbcC} is enforced through 
\eqref{eq:efweakQc},
and the same can be shown for the second condition in \eqref{eq:efbc0},
similarly to \eqref{eq:cfweakC}. 
In addition, \eqref{eq:efweakQb} together with the trial space
$\yspace$ yields the second condition in \eqref{eq:efbcN}. 
It remains to validate that \eqref{eq:efweakQa} weakly enforces 
the third condition in \eqref{eq:efbc0}. As before, this can be done
on closely following the argument in \cite[(A.10)--(A.19)]{hypbolpwf}, 
and collecting the terms that appear on $\partial_0 I$ due to integration by 
parts. In fact, \cite[(A.13), (A.14)]{hypbolpwf} yield the boundary 
term $(\varkappa_g)_s - \tfrac12\,\varkappa_g\,(\ln g(\vec x))_s
+ (g^\frac12(\vec x)\,\varkappa - g(\vec x)\,\varkappa_g)
\,\vec y_g\,.\,\vec\tau$ on $\partial_0 I$, which thanks
to $\vec y_g \in \yspace$ and $\vec x_\rho\,.\,\vec\ek_1 = 0$ collapses to
enforcing \eqref{eq:kappags}. 

\setcounter{equation}{0}
\section{Semidiscrete finite element approximations} \label{sec:sd}

Let $[0,1]=\cup_{j=1}^J I_j$, $J\geq3$, be a
decomposition of $[0,1]$ into intervals given by the nodes $q_j$,
$I_j=[q_{j-1},q_j]$. 
For simplicity, and without loss of generality,
we assume that the subintervals form an equipartitioning of $[0,1]$,
i.e.\ that 
\begin{equation*} 
q_j = j\,h\,,\quad \mbox{with}\quad h = J^{-1}\,,\qquad j=0,\ldots, J\,.
\end{equation*}
Clearly, if $I=\RZ$ we identify $0=q_0 = q_J=1$. 
In addition, we let $q_{J+1}=q_1$.

The necessary finite element spaces are defined as follows:
\[
V^h = \{\chi \in C(\overline I) : \chi\!\mid_{I_j} 
\text{ is affine for}\ j=1,\ldots,J\}
\quad\text{and}\quad \Vh = [V^h]^2\,.
\]
In addition, we define $\Vhpartialzero = \Vh \cap \Vpartialzero$ and
$W^h_{\partial_0} = \{ \chi \in V^h : \chi(\rho) = 0
\quad \forall\ \rho \in \partial_0 I\}$, as well as
$\xspace^h = \xspace \cap \Vh$ and $\yspace^h = \yspace \cap \Vh$,
recall \eqref{eq:spaces}. 
We define the mass lumped $L^2$--inner product $(u,v)^h$,
for two piecewise continuous functions, with possible jumps at the 
nodes $\{q_j\}_{j=1}^J$, via
\begin{equation}
( u, v )^h = \tfrac12\,h\,\sum_{j=1}^J 
\left[(u\,v)(q_j^-) + (u\,v)(q_{j-1}^+)\right],
\label{eq:ip0}
\end{equation}
where we define
$u(q_j^\pm)=\underset{\delta\searrow 0}{\lim}\ u(q_j\pm\delta)$.
The definition (\ref{eq:ip0}) naturally extends to vector valued functions.
Moreover, 
let $(\cdot,\cdot)^\diamond$ denote a discrete $L^2$--inner product based on
some numerical quadrature rule. In particular, 
for two piecewise continuous functions, with possible jumps at the 
nodes $\{q_j\}_{j=1}^J$, we let 
$(u,v)^\diamond = I^\diamond(u\,v)$, where
\begin{equation} \label{eq:Idiamond}
I^\diamond(f) = h \sum_{j=1}^J \sum_{k = 1}^K w_k\,
f(\alpha_k\,q_{j-1} + (1-\alpha_k)\,q_j)\,,\quad
w_k > 0\,,\ \alpha_k \in [0,1]\,,\quad k = 1,\ldots,K\,,
\end{equation}
with $K\geq2$, $\sum_{k=1}^K w_k = 1$,
and with distinct $\alpha_k$, $k=1,\ldots,K$.
A special case is
$(\cdot,\cdot)^\diamond = (\cdot,\cdot)^h$, recall
(\ref{eq:ip0}), but we also allow for more accurate quadrature rules.

Let $(\vec X^h(t))_{t\in[0,T]}$, with $\vec X^h(t)\in \Vh$, 
be an approximation to $(\vec x(t))_{t\in[0,T]}$. Then, 
similarly to (\ref{eq:tau}), we set
\begin{equation} \label{eq:tauh}
\vec\tau^h = \vec X^h_s = \frac{\vec X^h_\rho}{|\vec X^h_\rho|} 
\qquad \mbox{and} \qquad \vec\nu^h = -(\vec\tau^h)^\perp\,.
\end{equation}
For later use, we let $\vec\omega^h \in \Vh$ be the mass-lumped 
$L^2$--projection of $\vec\nu^h$ onto $\Vh$, i.e.\
\begin{equation} \label{eq:omegah}
\left(\vec\omega^h, \vec\varphi \, |\vec X^h_\rho| \right)^h 
= \left( \vec\nu^h, \vec\varphi \, |\vec X^h_\rho| \right)
= \left( \vec\nu^h, \vec\varphi \, |\vec X^h_\rho| \right)^h
\qquad \forall\ \vec\varphi\in\Vh\,.
\end{equation}

\subsection{Curvature flow}

We consider the following finite element approximation of 
$(\BGNmckappa)$, recall \eqref{eq:cfweakA}. It is closely related to the
approximation \cite[(3.3), (3.10)]{hypbol} for closed curve evolutions.

\noindent
$(\BGNmckappa_h)^h$:
Let $\vec X^h(0) \in \Vhpartialzero$. For $t \in (0,T]$, 
find $(\vec X^h(t), \kappa^h(t)) \in \Vh \times V^h$,
such that $\vec X^h_t(t) \in \xspace^h$, and such that
\begin{subequations} \label{eq:sd}
\begin{align} \label{eq:sda}
& \left(g(\vec X^h)\,\vec X^h_t, \chi\,\vec\omega^h\,|\vec X^h_\rho|\right)^h
=\left( \mathcal K(\kappa^h, \vec\omega^h, \vec X^h), 
\chi\,|\vec X^h_\rho|\right)^h
\quad \forall\ \chi \in V^h\,, \\ &
\left(\kappa^h\,\vec\omega^h, \vec\eta\,|\vec X^h_\rho|\right)^h
+ \left(\vec X^h_\rho, \vec\eta_\rho\,|\vec X^h_\rho|^{-1}\right) 
= 0 \quad \forall\ \vec\eta \in \xspace^h\,,
\label{eq:sdb}
\end{align}
\end{subequations}
where we have defined
$\mathcal K(\kappa^h, \vec\omega^h, \vec X^h) \in V^h$ by
\begin{equation} \label{eq:mathcalK}
\mathcal K(\kappa^h, \vec\omega^h, \vec X^h)(q_j) = 
\begin{cases}
\kappa^h(q_j) - \tfrac12\,\vec\omega^h(q_j)\,.\,\nabla\,\ln g(\vec X^h(q_j))
& q_j \in \overline I \setminus \partial_0 I\,, \\
\begin{cases}
(1 - \mu)\,\kappa^h(q_j) & \eqref{eq:gmu}\,, \\
n\,\kappa^h(q_j) + \tfrac12\,\vec X^h(q_j)\,.\,\vec\omega^h(q_j)
 & \eqref{eq:gAngenent} \,,
\end{cases} 
& q_j \in \partial_0 I\,.
\end{cases}
\end{equation}
To motivate the choice \eqref{eq:mathcalK}, we observe that 
it follows from \eqref{eq:nue1}, Table~\ref{tab:g} and L'Hospital's rule
that
\[
\left[
\varkappa -\tfrac12\,  \vec{\nu}\,.\, \nabla\, \ln g(\vec{x}) \right]
\to
\begin{cases}
(1-\mu)\,\varkappa & \eqref{eq:gmu}\,,\\
n\,\varkappa + \tfrac12\,(\vec x\,.\,\vec\ek_2)\,\vec\nu\,.\,\vec\ek_2
= n\,\varkappa + \tfrac12\,\vec x\,.\,\vec\nu& \eqref{eq:gAngenent} \,,
\end{cases}
\ \text{as $\rho \to \rho_0 \in \partial_0 I$\,.} 
\]
On recalling \eqref{eq:omegah}, we note that replacing $\vec\omega^h$ with
$\vec\nu^h$ in \eqref{eq:sdb} does not change the scheme. On the other hand,
for nonconstant $g$ we must use $\vec\omega^h$ in the first term in 
\eqref{eq:sda} to allow for an existence and uniqueness proof on the fully
discrete level, see Lemma~\ref{lem:ex} below. In addition, as 
\eqref{eq:mathcalK} is defined vertex-wise, we use the vertex normals
$\vec\omega^h$ there.

It does not appear possible to prove a stability result for the scheme
\eqref{eq:sd}. However, thanks to \eqref{eq:sdb} the scheme
$(\BGNmckappa_h)^h$ satisfies a weak equidistribution property, i.e.\ it can be
shown that neighbouring elements of $\Gamma^h(t) = \vec X^h(\overline I,t)$ 
have the same length if they are not parallel.
In effect, the side condition \eqref{eq:sdb} induces some tangential
motion, which ensures that the equidistribution property holds at all times.
The authors first introduced and discussed schemes with such an implicit
tangential motion in \cite{triplej} and have used it in a series of works
since.
We refer to the recent
review article \cite{bgnreview} for more details on that aspect of the scheme.

As an alternative approximation, we propose the following finite element 
approximation of $(\BGNmc)$, recall \eqref{eq:cfweakC}. 
It is the natural extension to the open curve case of the 
semidiscrete analogue of \cite[(3.5), (3.18)]{hypbol}.

\noindent
$(\BGNmc_{h})^{h}$:
Let $\vec X^h(0) \in \Vhpartialzero$. For $t = (0,T]$, 
find $(\vec X^h(t), \kappa_g^h(t)) \in \Vh \times W^h_{\partial_0}$,
such that $\vec X^h_t(t) \in \xspace^h$, and such that
\begin{subequations} \label{eq:sdnew}
\begin{align} \label{eq:sdnewa}
& \left(g(\vec X^h)\,\vec X^h_t, \chi\,\vec\nu^h\,
|\vec X^h_\rho|\right)^{h}
= \left(g^\frac12(\vec X^h)\,\kappa_g^h,
\chi\,|\vec X^h_\rho|\right)^{h} \quad \forall\ \chi \in W^h_{\partial_0}\,,\\
& \left(g(\vec X^h)\,\kappa_g^h\,\vec\nu^h,
\vec\eta\,|\vec X^h_\rho|\right)^{h}
+ \left( \nabla\,g^\frac12(\vec X^h) ,
\vec\eta \, |\vec X^h_\rho| \right)^{h}
+ \left( g^\frac12(\vec X^{h})\,
\vec X^h_\rho,\vec\eta_\rho\, |\vec X^h_\rho|^{-1} \right)^{h}
= 0 
\nonumber \\ & \hspace{10cm}
\quad \forall\ \vec\eta \in \xspace^h\,.
\label{eq:sdnewb}
\end{align}
\end{subequations}
Observe that we fix $\kappa_g^h(t)$ to be zero on $\partial_0 I$, since these
values can be arbitrary. Setting them to zero allows for a uniqueness result
on the fully discrete level. Note furthermore that on replacing 
$\vec\nu^h$ with $\vec\omega^h$ we obtain a related, but different,
approximation of $(\BGNmc)$, that will have the analogous 
theoretical properties.
We prefer the more natural variant as stated, in agreement with the closed
curve scheme from \cite{hypbol}. 
Similarly to $(\BGNmckappa_h)^h$, the
scheme $(\BGNmc_{h})^{h}$ also exhibits some implicit tangential motion,
but, in contrast to $(\BGNmckappa_h)^h$, it does not appear possible to
derive rigorous results on it. However, the scheme does admit a stability
bound. To formulate this result, and on recalling (\ref{eq:Lg}), 
for $\vec Z \in \Vh$ we let
\begin{equation} \label{eq:Lgh}
L_g^{h}(\vec Z) = \left( g^\frac12(\vec Z),|\vec Z_\rho| \right)^{h} .
\end{equation}
Then we can prove the following discrete analogue of (\ref{eq:gL2gradflow}) 
for the scheme \eqref{eq:sdnew}.  

\begin{thm} \label{thm:cfstab}
Let $(\vec X^h(t),\kappa^h_g(t)) \in \Vh\times W^h_{\partial_0}$, 
for $t\in (0,T]$,
be a solution to $(\BGNmc_h)^h$.
Then the solution satisfies the stability bound 
\begin{equation} \label{eq:gL2gradflowh}
\ddt L_g^h(\vec X^h(t)) + \left(g^\frac12(\vec X^h)\,\kappa_g^h,
\kappa_g^h\,|\vec X^h_\rho|\right)^{h} = 0\,.
\end{equation}
\end{thm}
\begin{proof}
Choosing $\chi = \kappa_g^h \in W^h_{\partial_0}$ in \eqref{eq:sdnewa} 
and $\vec\eta = \vec X^h_t \in \xspace^h$ in \eqref{eq:sdnewb} yields 
\begin{equation*}
\left( \nabla\,g^\frac12(\vec X^h) ,
\vec X^h_t \, |\vec X^h_\rho| \right)^{h}
+ \left( g^\frac12(\vec X^{h})\,
\vec X^h_\rho,\vec X^h_{\rho,t}\, |\vec X^h_\rho|^{-1} \right)^{h}
+  \left(g^\frac12(\vec X^h)\,\kappa_g^h,
\kappa_g^h\,|\vec X^h_\rho|\right)^{h} = 0\,,
\end{equation*}
and so we obtain the desired result \eqref{eq:gL2gradflowh}. 
\end{proof}

The identity \eqref{eq:gL2gradflowh}, after integration in time, yields
a stability bound for the discrete length, and it is a direct discrete analogue
of \eqref{eq:gL2gradflow}.

\subsection{Elastic flow}

Following the approach by the authors in \cite{hypbolpwf}, it is
straightforward to derive a semidiscrete approximation of
$(\BGNpwf)$, in the case that $\partial_0 I = \emptyset$.
The derived scheme, which will correspond to 
\cite[$(\BGNpwf_h)^h$]{hypbolpwf} with the natural changes to the test and
trial spaces, can be shown to be stable.
Moreover, and similarly to $(\BGNmckappa_h)^h$, the derived scheme will 
satisfy an equidistribution property. 
However, as it appears to be highly nontrivial to
extend the approximation to
the case $\partial_0 I \not= \emptyset$, we do not pursue this variant
any further in this paper.

On the other hand, the following discretisation based on the formulation
$(\BGNpwfwf)$ can naturally deal with all the considered boundary conditions.
It is the natural extension to the open curve case of the 
semidiscrete scheme \cite[$(\BGNpwfwf_h)^\diamond$]{hypbolpwf}.

\noindent
$(\BGNpwfwf_h)^\diamond$:
Let $\vec X^h(0) \in \Vhpartialzero$. 
For $t \in (0,T]$, find $(\vec X^h(t), \kappa^h_g(t), \vec Y^h_g(t)) 
\in \Vh \times V^h \times \yspace^h$,
with $\vec X^h_t(t) \in \xspace^h$, such that
\begin{subequations} \label{eq:B88}
\begin{align} 
& \left(g(\vec X^h)\,\vec X^h_t\,.\,
\vec\omega^h, \vec\chi\,.\,\vec\omega^h\,
|\vec X^h_\rho|_g \right)^\diamond
- \left((\vec Y^h_g)_s\,.\,\vec\nu^h, \vec\chi_s\,.\,\vec\nu^h
\,|\vec X^h_\rho|_g \right)^\diamond
\nonumber \\ & \quad
 = -\tfrac12 \left( (\kappa^h_g\,)^2 
- \vec Y^h_g\,.\,\nabla\,\ln\,g(\vec X^h), \left[\vec\chi_s\,.\,\vec\tau^h + 
\tfrac12\,\vec\chi\,.\,\nabla\,\ln\,g(\vec X^h)\right] |\vec X^h_\rho|_g
 \right)^\diamond \nonumber \\ & \qquad
+ \tfrac12 \left((D^2\,\ln\,g(\vec X^h))\,\vec Y^h_g, \vec\chi\,
|\vec X^h_\rho|_g \right)^\diamond \nonumber \\ & \qquad
+  \left( g^\frac12(\vec X^h)\,\kappa^h_g\,\vec Y^h_g\,.\,\vec\nu^h
+ \tfrac12\,(\vec Y^h_g)_s\,.\,\vec\tau^h, \vec\chi\,.\,
(\nabla\,\ln\,g(\vec X^h))\,
|\vec X^h_\rho|_g \right)^\diamond \nonumber \\ & \qquad
+ \left( g^\frac12(\vec X^h)\,\kappa^h_g,
\vec\chi_s\,.\,(\vec Y^h_g)^\perp\,|\vec X^h_\rho|_g \right)^\diamond
\quad \forall\ \vec\chi \in \xspace^h\,, \label{eq:B88a} \\
&\left(\kappa^h_g - g^\frac12(\vec X^h)\,\vec Y^h_g\,.\,\vec\nu^h, 
\chi\,|\vec X^h_\rho|_g \right)^\diamond = 0
\quad \forall\ \chi \in V^h\,, \label{eq:B88b} \\
&
\left(g^\frac12(\vec X^h)\,\kappa_g^h\,\vec\nu^h, 
\vec\eta\,|\vec X^h_\rho|_g \right)^\diamond
+ \left(\vec X^h_s,\vec\eta_s\,|\vec X^h_\rho|_g\right)^\diamond
+ \tfrac12 \left( \nabla\,\ln\,g(\vec X^h), \vec\eta\,|\vec
X^h_\rho|_g\right)^\diamond 
\nonumber \\ & \quad
= \sum_{p \in \partial_C I} [g^{\frac12}(\vec X^h)\,\vec\zeta\,.\,\vec\eta](p)
\quad \forall\ \vec\eta \in \yspace^h\,. \label{eq:B88c}
\end{align}
\end{subequations}
If we replace $\vec\omega^h$ with $\vec\nu^h$ in \eqref{eq:B88a} we obtain a
related, but different, approximation of $(\BGNpwfwf)$ with the analogous
theoretical properties. We prefer the scheme as stated because it simplifies
the presentation of the existence and uniqueness proof on the fully
discrete level, see Lemma~\ref{lem:exg} below. In addition, as stated the
scheme is consistent with the closed curve variant from \cite{hypbolpwf}.

We have the following discrete analogue of \eqref{eq:eL2gradflow}. 

\begin{thm} \label{thm:stabg}
Let $(\vec X^h(t),\kappa^h_g(t),\vec Y_g^h(t)) 
\in \Vh\times V^h \times \yspace^h$, for $t\in (0,T]$,
be a solution to $(\BGNpwfwf_h)^\diamond$. Then the solution satisfies 
\begin{equation} \label{eq:Qhstab}
\tfrac12\,\ddt
\left((\kappa^h_g)^2 
, |\vec X^h_\rho|_g\right)^\diamond 
+ \left(g(\vec X^h)\,(\vec X^h_t\,.\,\vec\omega^h)^2, 
|\vec X^h_\rho|_g \right)^\diamond 
= 0\,.
\end{equation}
\end{thm}
\begin{proof}
Similarly to the proof of \cite[Theorem~4.4]{hypbolpwf}, 
on choosing $\vec\chi = \vec X^h_t \in \xspace^h$ in \eqref{eq:B88a} 
we can show that
\begin{align}
& \left(g(\vec X^h)\,(\vec X^h_t\,.\,\vec\omega^h)^2, 
|\vec X^h_\rho|_g \right)^\diamond \nonumber \\ & \quad 
= - \tfrac12 \left( (\kappa^h_g)^2 
, (|\vec X^h_\rho|_g)_t \right)^\diamond
+ \left( \kappa^h_g\,\vec Y^h_g, 
(g^\frac12(\vec X^h)\, \vec\nu^h \,|\vec X^h_\rho|_g)_t \right)^\diamond
\nonumber \\ & \qquad
+ \left( (\vec Y^h_g)_\rho,  (g^\frac12(\vec X^h)\,\vec\tau^h)_t 
\right)^\diamond
+ \tfrac12 \left( \vec Y^h_g, 
((\nabla\,\ln\,g(\vec X^h))\,|\vec X^h_\rho|_g )_t \right)^\diamond .
\label{eq:B5Xt}
\end{align}
Moreover, differentiating \eqref{eq:B88c} with respect to time, 
noting that $\vec X^h_t = 0$ on $\partial_C I$, 
and then choosing $\vec\eta = \vec Y^h_g$, yields that
\begin{align} \label{eq:B3dt}
& 
\left((\kappa_g^h)_t\,\vec Y^h_g, g^\frac12(\vec X^h)\,\vec\nu^h\, 
|\vec X^h_\rho|_g\right)^\diamond
+\left(\kappa_g^h\,\vec Y^h_g, (g^\frac12(\vec X^h)\,\vec\nu^h\, 
|\vec X^h_\rho|_g)_t \right)^\diamond
\nonumber \\ & \qquad
+ \left((\vec Y^h_g)_\rho,(g^\frac12(\vec X^h)\,\vec\tau^h)_t\right)^\diamond
+ \tfrac12 \left(\vec Y^h_g, ((\nabla\,\ln\,g(\vec X^h))\, 
|\vec X^h_\rho|_g)_t \right)^\diamond
= 0 .
\end{align}
Finally, choosing $\chi = (\kappa_g^h)_t$ in \eqref{eq:B88b}, 
and combining with (\ref{eq:B5Xt}) and (\ref{eq:B3dt}),
yields the desired result (\ref{eq:Qhstab}).
\end{proof}

The identity \eqref{eq:Qhstab}, after integration in time, yields
a stability bound for the discrete elastic energy.

\setcounter{equation}{0}
\section{Fully discrete finite element approximations} \label{sec:fd}

Let $0= t_0 < t_1 < \ldots < t_{M-1} < t_M = T$ be a
partitioning of $[0,T]$ into possibly variable time steps 
$\ttau_m = t_{m+1} - t_{m}$, $m=0,\ldots, M-1$. 
For a given $\vec{X}^m\in \Vh$ we let
$\vec\nu^m$ and $\vec\omega^m$ be the fully discrete
analogues to (\ref{eq:tauh}) and \eqref{eq:omegah}, respectively. 

For the implementation of the presented schemes some metric-dependent 
quantities need to be calculated. For the metrics in \eqref{eq:gs} and
\eqref{eq:gAngenent}, \eqref{eq:gcone} we list these expressions for the
convenience of the reader in Table~\ref{tab:g}.
For the metric \eqref{eq:gGNS} all the necessary quantities can be calculated
with the help of the chain rule, using the fact that e.g.\
$(\nabla\,g)(\vec z) = 
U^T\, (\nabla_{\bf u}\,\Psi)({\bf u}_0 + U\,\vec z)$.
\begin{table}
\center
\def\arraystretch{1.25}
\begin{tabular}{|c|c|c|}
\hline
$g$ & 
$\tfrac12\,\nabla\,\ln g(\vec x)$ & $\tfrac12\,D^2\,\ln g(\vec x)$ 
\\ \hline
(\ref{eq:gmu}) & 
$- \frac{\mu}{\vec x\,.\,\vec\ek_1}\,\vec\ek_1$ &
$\frac{\mu}{(\vec x\,.\,\vec\ek_1)^2}\,\vec\ek_1 \otimes \vec\ek_1$
\\
(\ref{eq:galpha}) & 
$\frac{2\,\alpha}{1 - \alpha\,|\vec x|^2}\,\vec x$ & 
$\frac{2\,\alpha}{1 - \alpha\,|\vec x|^2}\,\mat \Id
+ \frac{4\,\alpha^2}{(1 - \alpha\,|\vec x|^2)^2}\,\vec x \otimes \vec x$
\\ 
(\ref{eq:gMercator}) & $-\tanh(\vec x\,.\,\vec\ek_1)\,\vec\ek_1$
& $-\cosh^{-2}(\vec x\,.\,\vec\ek_1)\,\vec\ek_1 \otimes \vec\ek_1$
\\
(\ref{eq:gcatenoid}) & $\tanh(\vec x\,.\,\vec\ek_1)\,\vec\ek_1$
& $\cosh^{-2}(\vec x\,.\,\vec\ek_1)\,\vec\ek_1 \otimes \vec\ek_1$
\\
(\ref{eq:gtorus}) & $-\frac{\sin(\vec x\,.\,\vec\ek_2)}
{[\mathfrak s^2 + 1]^\frac12 - \cos(\vec x\,.\,\vec\ek_2)}\,\vec\ek_2$
& $\frac{1 - [\mathfrak s^2 + 1]^\frac12\,\cos (\vec x\,.\,\vec\ek_2)}
{([\mathfrak s^2 + 1]^\frac12 - \cos (\vec x\,.\,\vec\ek_2))^2}
\,\vec\ek_2 \otimes \vec\ek_2$ \\
(\ref{eq:gAngenent}) & 
$\frac{n-1}{\vec x\,.\,\vec\ek_1}\,\vec\ek_1-\tfrac12\,\vec x$
& $- \frac{n-1}{(\vec x\,.\,\vec\ek_1)^2}
\,\vec\ek_1 \otimes \vec\ek_1 - \tfrac12\,\mat\Id$ \\
\eqref{eq:gcone} & 
$\mathfrak{b}\,\vec\ek_1$ & 
$\mat 0$ 
\\ \hline \hline
& 
$\nabla\,g^\frac12(\vec x)$ & $\nabla\,g^\frac12_-(\vec x)$ \\ \hline
(\ref{eq:gmu}) & $- \frac\mu{(\vec x\,.\,\vec\ek_1)^{\mu+1}}\,\vec\ek_1$ &
$0$ \\
(\ref{eq:galpha}) & 
$\frac{4\,\alpha}{( 1 - \alpha\, |\vec x|^2)^{2}}\,\vec x$ & 
$4\, \min\{0,\alpha\}\,\vec x$ \\ 
(\ref{eq:gMercator}) & 
$-\frac{\tanh(\vec x\,.\,\vec\ek_1)}{\cosh(\vec x\,.\,\vec\ek_1)}\,\vec\ek_1$ & 
$- (\vec x\,.\,\vec\ek_1)\,\vec\ek_1$ 
\\
(\ref{eq:gcatenoid}) & $\sinh(\vec x\,.\,\vec\ek_1)\,\vec\ek_1$ &
$0$ \\
(\ref{eq:gtorus}) & $-\frac{\mathfrak s\,\sin(\vec x\,.\,\vec\ek_2)}
{([\mathfrak s^2 + 1]^\frac12 - \cos(\vec x\,.\,\vec\ek_2))^2}\,\vec\ek_2$ & 
$- \frac{\mathfrak s\,\vec x\,.\,\vec\ek_2}
{([\mathfrak s^2 + 1]^\frac12 - 1)^{2}}\,\vec\ek_2$  \\
(\ref{eq:gAngenent}) & 
\makecell{
$e^{-\frac14\, |\vec z|^2}\bigl[
(\vec x\,.\,\vec\ek_1)^{n-2}
\left(n-1-\tfrac12\,(\vec x\,.\,\vec\ek_1)^2\right)\vec\ek_1$ \\
$-\tfrac12\,(\vec x\,.\,\vec\ek_1)^{n-1}\,(\vec x\,.\,\vec\ek_2)\,\vec\ek_2\bigr]$
} &
$- 1.29\,\vec x$ \ ($n=2$) \\
\eqref{eq:gcone} & 
$\frac{\mathfrak{b}^2}{[1 - \mathfrak{b}^2]^\frac12}\,
e^{\mathfrak{b}\,\vec z\,.\,\vec\ek_1}\,\vec\ek_1$ &
0 \\ \hline
\end{tabular}
\caption{Expressions for terms that are relevant for the implementation of the
presented finite element approximations. 
Observe that 
$\nabla\,g^\frac12 = g^\frac12\,(\tfrac12\,\nabla\,\ln g)$ and
$\nabla\,g^\frac12_+ = \nabla\,g^\frac12 - \nabla\,g^\frac12_-$.}
\label{tab:g}
\end{table}%

\subsection{Curvature flow}

We consider the following linear fully discrete analogue of 
$(\BGNmckappa_h)^h$.

\noindent
$(\BGNmckappa_m)^h$:
Let $\vec X^0 \in \Vhpartialzero$. For $m=0,\ldots,M-1$, 
find $(\vec X^{m+1}, \kappa^{m+1}) \in \Vh \times V^h$,
with $\vec X^{m+1} - \vec X^m \in \xspace^h$, such that
\begin{subequations} \label{eq:fd}
\begin{align} \label{eq:fda}
& \left(g(\vec X^m)\,
\frac{\vec X^{m+1} - \vec X^m}{\ttau_m}, \chi\,\vec\omega^m\,|\vec
X^m_\rho|\right)^h
= \left(\mathcal K(\kappa^{m+1},\vec\omega^m,\vec X^m) , 
\chi\,|\vec X^m_\rho|\right)^h 
\quad \forall\ \chi \in V^h\,, \\ &
\left(\kappa^{m+1}\,\vec\omega^m, \vec\eta\,|\vec X^m_\rho|\right)^h
+ \left(\vec X^{m+1}_\rho, \vec\eta_\rho\,|\vec X^m_\rho|^{-1}\right) 
= 0 \quad \forall\ \vec\eta \in \xspace^h\,.
\label{eq:fdb}
\end{align}
\end{subequations}
Note that the scheme $(\BGNmckappa_m)^h$ is a natural generalisation of the
scheme \cite[$(\BGNmckappa_m)^h$]{hypbol} to the case of open curves.
The scheme $(\BGNmckappa_m)^h$ has the advantage that it is linear, 
recall \eqref{eq:mathcalK}, and that
it asymptotically inherits the equidistribution property
from $(\BGNmckappa_h)^h$, \eqref{eq:sd}. 
Observe that we have chosen the explicit and implicit terms in \eqref{eq:fd} 
such that we obtain a linear scheme for which existence and uniqueness can be
shown. The basic structure of the scheme goes back to
\cite[(2.3)]{triplejMC}. In fact, for a closed curve in the Euclidean plane
the scheme $(\BGNmckappa_m)^h$ collapses to that approximation of
curvature flow.

We make the following mild assumption.
\begin{tabbing}
$(\mathfrak A)^{h}$\quad \=
Let $|\vec{X}^m_\rho| > 0$ for almost all $\rho\in I$, and let
$\dim \spa \left\{ \vec\omega^m(q_j) : 
q_j \in \overline I\setminus \partial_0 I \right\} = 2$.
\end{tabbing}

\begin{lem} \label{lem:ex}
Let the assumption $(\mathfrak A)^h$ hold.
Then there exists a unique solution 
$(\vec X^{m+1},$ $ \kappa^{m+1}) \in \Vh \times V^h$ to 
$(\BGNmckappa_m)^h$.
\end{lem}
\begin{proof}
As (\ref{eq:fda}), (\ref{eq:fdb}) is linear, recall \eqref{eq:mathcalK},
existence follows from uniqueness. 
To investigate the latter, we consider the system: 
Find $(\delta\vec X,\kappa) \in \xspace^h \times V^h$ such that
\begin{subequations}
\begin{align}
\left(g(\vec X^m)\,\frac{\delta\vec X}{\ttau_m}, \chi\,\vec\omega^m\,|\vec
X^m_\rho|\right)^h
= \left(\lambda\,\kappa , \chi\,|\vec X^m_\rho|\right)^h 
\quad \forall\ \chi \in V^h\,, \label{eq:proofa}\\
\left(\kappa\,\vec\omega^m, \vec\eta\,|\vec X^m_\rho|\right)^h
+ \left(\delta\vec X_\rho, \vec\eta_\rho\,|\vec X^m_\rho|^{-1}\right) = 0 
\quad \forall\ \vec\eta \in \xspace^h\,,
\label{eq:proofb}
\end{align}
\end{subequations}
where we recall from (\ref{eq:mathcalK}) that $\lambda \in V^h$ with
$\lambda > 0$ in $\overline I$. It immediately follows from
\eqref{eq:proofa} that $\kappa = 0$ on $\partial_0 I$, and so $\kappa \in
W^h_{\partial_0}$.
Choosing $\vec\eta = \delta\vec X \in \xspace^h$ in 
(\ref{eq:proofb}) and
$\chi=\hat \kappa \in W^h_{\partial_0}$ in (\ref{eq:proofa}), with
$\hat \kappa(q_j) = g^{-1}(\vec X^m(q_j))\,\kappa(q_j)$ for $q_j \in
\overline I \setminus \partial_0 I$,
yields that
\begin{align} \label{eq:unique0}
0& =
\left(|\delta\vec X_\rho|^2, |\vec X^m_\rho|^{-1}\right)
+\ttau_m \left(\lambda\,\kappa, \hat \kappa\, |\vec X^m_\rho|\right)^h 
\nonumber \\ & 
= \left(|\delta\vec X_\rho|^2, |\vec X^m_\rho|^{-1}\right)
+\ttau_m \left(\lambda\,g(\vec X^m)\, |\hat \kappa|^2,|\vec X^m_\rho|\right)^h 
.
\end{align}
It follows from (\ref{eq:unique0}) that $\kappa = \hat \kappa = 0$ and that
$\delta\vec X$ is constant. 
Hence \eqref{eq:proofa} and \eqref{eq:ip0} imply that
\begin{equation}
\left(g(\vec X^m)\,\delta\vec X, \chi\,\vec\omega^m\,
|\vec X^m_\rho|\right)^h = 0
 \quad \forall\ \chi \in V^h \,. \label{eq:unique1}
\end{equation}
It follows from (\ref{eq:unique1}) and
assumption $(\mathfrak A)^h$ that $\delta\vec X=\vec0$.
Hence we have shown that $(\BGNmckappa_m)^h$ has a unique solution
$(\vec X^{m+1},\kappa^{m+1}) \in \Vh\times V^h$.
\end{proof}

In order to present an unconditionally stable fully discrete approximation
of $(\BGNmc_{h})^{h}$, we assume that we can split $g^\frac12$ into
\begin{equation} \label{eq:gsplit}
g^\frac12 = g^\frac12_+ + g^\frac12_-
\quad\text{ such that $\,\pm g^\frac12_\pm\,$ is convex in $H$.}
\end{equation}
It follows from the splitting in (\ref{eq:gsplit}) that
\begin{equation} \label{eq:gsplitstab}
\nabla\,[g^\frac12_+(\vec u) + g^\frac12_-(\vec v)]\,.\,(\vec u - \vec v) \geq
g^\frac12(\vec u) - g^\frac12(\vec v) \quad \forall\ \vec u, \vec v \in H\,.
\end{equation}
Then we introduce the following nonlinear scheme.

\noindent
$(\BGNmc_{m,\star})^{h}$:
Let $\vec X^0 \in \Vhpartialzero$. For $m=0,\ldots,M-1$, 
find $(\vec X^{m+1}, \kappa_g^{m+1}) \in \Vh \times W^h_{\partial_0}$,
with $\vec X^{m+1} - \vec X^m \in \xspace^h$, such that
\begin{subequations} 
\begin{align} \label{eq:fdnewa}
& \left(g(\vec X^m)\,
\frac{\vec X^{m+1} - \vec X^m}{\ttau_m}, \chi\,\vec\nu^m\,
|\vec X^m_\rho|\right)^{h}
= \left(g^\frac12(\vec X^m)\,\kappa_g^{m+1},
\chi\,|\vec X^m_\rho|\right)^{h} \quad \forall\ \chi \in W^h_{\partial_0}\,,\\
& \left(g(\vec X^m)\,\kappa_g^{m+1}\,\vec\nu^m,
\vec\eta\,|\vec X^m_\rho|\right)^{h}
+ \left( \nabla\,[g^\frac12_+(\vec X^{m+1}) + g^\frac12_-(\vec X^{m})],
\vec\eta \, |\vec X^{m+1}_\rho| \right)^{h}
\nonumber \\ & \hspace{5cm}
+ \left( g^\frac12(\vec X^{m})\,
\vec X^{m+1}_\rho,\vec\eta_\rho\, |\vec X^m_\rho|^{-1} \right)^{h}
= 0 \quad \forall\ \vec\eta \in \xspace^h\,.
\label{eq:fdnonlinearb}
\end{align}
\end{subequations}
Note that the scheme $(\BGNmc_{m,\star})^h$ is a natural generalisation of the
scheme \cite[$(\BGNmc_{m,\star})^h$]{hypbol} to the case of open curves.

We can prove the following fully discrete analogue of 
Theorem~\ref{thm:cfstab}.

\begin{thm} 
Let $(\vec X^{m+1},\kappa_g^{m+1})$ be a solution to 
$(\BGNmc_{m,\star})^{h}$. Then it holds that
\begin{equation} \label{eq:stab}
L_g^{h}(\vec X^{m+1}) + \ttau_m
\left(g^\frac12(\vec X^m)\,|\kappa_g^{m+1}|^2,
|\vec X^m_\rho|\right)^{h} \leq L_g^{h}(\vec X^m)\,.
\end{equation}
\end{thm}
\begin{proof}
Choosing $\chi = \ttau_m\,\kappa_g^{m+1}\in  W^h_{\partial_0}$ in 
(\ref{eq:fdnewa}) and
$\vec\eta = \vec X^{m+1} - \vec X^m \in \xspace^h$ in (\ref{eq:fdnonlinearb}) 
yields that
\begin{align*}
& - \ttau_m \left(g^\frac12(\vec X^m)\,|\kappa_g^{m+1}|^2,
|\vec X^m_\rho|\right)^{h}
= \left( \nabla\,[g^\frac12_+(\vec X^{m+1}) + g^\frac12_-(\vec X^{m})],
(\vec X^{m+1} - \vec X^m) \, |\vec X^{m+1}_\rho| \right)^{h}
\nonumber \\ & \qquad
+ \left( g^\frac12(\vec X^{m})\,
\vec X^{m+1}_\rho,(\vec X^{m+1}_\rho - \vec X^m_\rho)\, 
|\vec X^m_\rho|^{-1} \right)^{h}
\nonumber \\ & 
\geq \left( g^\frac12(\vec X^{m+1}) - g^\frac12(\vec X^{m}), 
|\vec X^{m+1}_\rho| \right)^{h}
+ \left( g^\frac12(\vec X^{m}), |\vec X^{m+1}_\rho| - |\vec X^m_\rho|
\right)^{h}
\nonumber \\ &
= \left( g^\frac12(\vec X^{m+1})\,|\vec X^{m+1}_\rho| - 
g^\frac12(\vec X^{m})\,|\vec X^m_\rho|, 1 \right)^{h} 
= L_g^{h}(\vec X^{m+1}) - L_g^{h}(\vec X^{m})\,, 
\end{align*}
where we have used (\ref{eq:gsplitstab}) and the inequality
$\vec a\,.\,(\vec a - \vec b) \geq |\vec b|\,(|\vec a| - |\vec b|)$
for $\vec a$, $\vec b \in \bR^2$.
\end{proof}

Splittings of the form \eqref{eq:gsplit} for the metrics \eqref{eq:gs} 
have been derived in \cite{hypbol}, and we repeat them for the benefit of the
reader in Table~\ref{tab:g}. In the same table we also list, where possible,
such splittings for the metrics \eqref{eq:gnew}.
In particular, for the metric \eqref{eq:gcone} we note that
$D^2\,g^\frac12(\vec z) = 
\frac{\mathfrak{b}^3}{[1 - \mathfrak{b}^2]^\frac12}\,
e^{\mathfrak{b}\,\vec z\,.\,\vec\ek_1}\,\vec\ek_1\otimes\vec\ek_1$
is clearly positive semidefinite, and so we can choose 
$g^\frac12_+ = g^\frac12$ and $g^\frac12_- = 0$.
Moreover, we now demonstrate how to obtain a splitting of the form
\eqref{eq:gsplit} for the metric
\eqref{eq:gAngenent} in the case $n=2$. We leave the case $n\geq3$ to the
reader. If $n=2$, then we note that
\[
D^2\,g^\frac12(\vec z) = \tfrac12\,
e^{-\frac14\, |\vec z|^2} 
\left[ \tfrac12\,(\vec z\,.\,\vec\ek_1)\,\vec z \otimes \vec z -
\begin{pmatrix}
 3\,\vec z\,.\,\vec\ek_1 & \vec z\,.\,\vec\ek_2 \\
 \vec z\,.\,\vec\ek_2 & \vec z\,.\,\vec\ek_1
\end{pmatrix} 
\right],
\]
and we observe that the eigenvalues of 
$\begin{pmatrix}
3\,\vec z\,.\,\vec\ek_1 & \vec z\,.\,\vec\ek_2 \\
\vec z\,.\,\vec\ek_2 & \vec z\,.\,\vec\ek_1
\end{pmatrix}$ 
are $2\,\vec z\,.\,\vec\ek_1 \pm |\vec z|$. Moreover, the function
$\mathcal{F}(\vec z) = \tfrac12\,e^{-\frac14\, |\vec z|^2} \,
(2\,\vec z\,.\,\vec\ek_1 + |\vec z|)$ attains its maximum at 
$\vec z = \sqrt{2}\,\vec\ek_1$ with 
$\max_{\vec z \in \bR^2} \mathcal{F}(\vec z) 
= \frac{3}{\sqrt{2\,e}} \approx 1.2866$.
Hence the matrix
\[
D^2\,g^\frac12(\vec z) + R\,\mat\Id\,,\quad \text{where } R=1.29\,,
\]
is positive definite for all $\vec z \in H$, and so we can choose
\[
g^\frac12_+(\vec z) = g^\frac12(\vec z) + \tfrac12\,R\,|\vec z|^2
\quad\text{and}\quad
g^\frac12_-(\vec z) = -\tfrac12\,R\,|\vec z|^2
\quad \text{if }\ n = 2\,.
\]

\subsection{Elastic flow}

We consider the following linear fully discrete analogue of the scheme
$(\BGNpwfwf_h)^\diamond$, \eqref{eq:B88}. 

\noindent
$(\BGNpwfwf_m)^\diamond$:
Let $(\vec X^0,\kappa^0_g,\vec Y^0_g) \in \Vhpartialzero\times V^h\times \Vh$. 
For $m=0,\ldots,M-1$, 
find $(\vec X^{m+1}, \kappa^{m+1}_g,$ $ \vec Y^{m+1}_g) 
\in \Vh \times V^h \times \yspace^h$,
with $\vec X^{m+1} - \vec X^m \in \xspace^h$, such that
\begin{subequations} \label{eq:B8}
\begin{align} 
& \left(g(\vec X^m)\,\frac{\vec X^{m+1} - \vec X^m}{\ttau_m}\,.\,
\vec\omega^m, \vec\chi\,.\,\vec\omega^m\,
|\vec X^m_\rho|_g \right)^\diamond
- \left((\vec Y^{m+1}_g)_s, \vec\chi_s\,|\vec X^m_\rho|_g \right)^\diamond
\nonumber \\ & \qquad
+ \left((\vec Y^m_g)_s\,.\,\vec\tau^m, \vec\chi_s\,.\,\vec\tau^m
\,|\vec X^m_\rho|_g \right)^\diamond
\nonumber \\ & 
 = -\tfrac12 \left( (\kappa^m_g\,)^2 
- \vec Y^m_g\,.\,\nabla\,\ln\,g(\vec X^m), \left[\vec\chi_s\,.\,\vec\tau^m + 
\tfrac12\,\vec\chi\,.\,\nabla\,\ln\,g(\vec X^m)\right] |\vec X^m_\rho|_g
 \right)^\diamond \nonumber \\ & \qquad
+ \tfrac12 \left((D^2\,\ln\,g(\vec X^m))\,\vec Y^m_g, \vec\chi\,
|\vec X^m_\rho|_g \right)^\diamond \nonumber \\ & \qquad
+  \left( g^\frac12(\vec X^m)\,\kappa^m_g\,\vec Y^m_g\,.\,\vec\nu^m
+ \tfrac12\,(\vec Y^m_g)_s\,.\,\vec\tau^m, \vec\chi\,.\,
(\nabla\,\ln\,g(\vec X^m))\,
|\vec X^m_\rho|_g \right)^\diamond \nonumber \\ & \qquad
+ \left( g^\frac12(\vec X^m)\,\kappa^m_g,
\vec\chi_s\,.\,(\vec Y^m_g)^\perp\,|\vec X^m_\rho|_g \right)^\diamond
\quad \forall\ \vec\chi \in \xspace^h\,, \label{eq:B8a} \\
&\left(\kappa^{m+1}_g - g^\frac12(\vec X^m)\,\vec Y^{m+1}_g\,.\,\vec\nu^m, 
\chi\,|\vec X^m_\rho|_g \right)^\diamond = 0
\quad \forall\ \chi \in V^h\,, \label{eq:B8b} \\
&
\left(g^\frac12(\vec X^m)\,\kappa_g^{m+1}\,\vec\nu^m, 
\vec\eta\,|\vec X^m_\rho|_g \right)^\diamond
+ \left(\vec X^{m+1}_s,\vec\eta_s\,|\vec X^m_\rho|_g\right)^\diamond
+ \tfrac12 \left( \nabla\,\ln\,g(\vec X^m), \vec\eta\,|\vec
X^m_\rho|_g\right)^\diamond 
\nonumber \\ & \quad
= \sum_{p \in \partial_C I} [g^{\frac12}(\vec X^m)\,\vec\zeta\,.\,\vec\eta](p)
\quad \forall\ \vec\eta \in \yspace^h\,. \label{eq:B8c}
\end{align}
\end{subequations}
Note that the scheme $(\BGNpwfwf_{m})^\diamond$ is a natural generalisation of 
the scheme \cite[$(\BGNpwfwf_{m})^\diamond$]{hypbolpwf} 
to the case of open curves.
Observe that the second term in \eqref{eq:B88a} is approximated by
the last two terms on the left hand side of \eqref{eq:B8a}. This is done in
order to allow for an existence and uniqueness proof, see Lemma~\ref{lem:exg}
below. In particular, the spatial differential operators acting on
$\vec Y^{m+1}_g$ and $\vec X^{m+1}$ in \eqref{eq:B8a} and \eqref{eq:B8c},
respectively, are now the same. This technique is in line with the authors'
earlier work in e.g.\ \cite{pwf,pwftj,hypbolpwf,axipwf}. 

We make the following mild assumptions.
\begin{tabbing}
$(\mathfrak B)^\diamond$\quad \=
Let $|\vec{X}^m_\rho| > 0$ for almost all $\rho\in I$, and let
$\dim \spa \mathcal Z^\diamond = 2$, where \\ \> $\mathcal Z^\diamond = 
\left\{ \left( g^\frac12(\vec X^m)\,\vec\nu^m, 
\chi\, |\vec X^m_\rho|_g \right)^\diamond : \chi \in V^h \right \} 
\subset \bR^2$.
\end{tabbing}
In the case $(\cdot,\cdot)^\diamond = (\cdot,\cdot)^h$ the above assumption
collapses to $(\mathfrak A)^h$. When dealing with clamped boundary conditions,
we also need the following assumption, which is similar
to \cite[Assumption~5.9]{axipwf}.

\begin{tabbing} 
$(\mathfrak C)^\diamond$\quad \=
If $\vec Z \in \yspace^h$ with
$( \vec Z_s, \vec\chi_s \,|\vec X^m_\rho|_g )^\diamond = 0$
for all $\vec\chi \in \xspace^h$ and \\ \>
$( g^\frac12(\vec X^m)\,\vec Z, \chi\,\vec\nu^m \,|\vec X^m_\rho|_g
)^\diamond = 0$ for all $\chi \in V^h$, 
then $\vec Z = \vec 0$.
\end{tabbing}

\begin{lem} \label{lem:exg}
Let the assumptions $(\mathfrak A)^h$ and 
$(\mathfrak B)^\diamond$ hold.
Moreover, if $\partial_C I \not= \emptyset$ then let
assumptions $(\mathfrak C)^\diamond$ hold.
Then there exists a unique solution
$(\vec X^{m+1}, \kappa^{m+1}_g, \vec Y^{m+1}_g) \in \Vh \times V^h \times 
\yspace^h$ to $(\BGNpwfwf_m)^\diamond$, if the quadrature rule
\eqref{eq:Idiamond} has at least one interior sampling point, 
$\alpha_k\in(0,1)$.
If $(\cdot,\cdot)^\diamond = (\cdot,\cdot)^h$, on the other hand, 
then there exists a solution 
that can be made
unique by requiring that $\kappa^{m+1}_g \in W^h_{\partial_0}$.
\end{lem}
\begin{proof}
We first consider the case that \eqref{eq:Idiamond} is such that
$\alpha_k \in(0,1)$ for some $1\leq k\leq K$.
As (\ref{eq:B8}) is linear, existence follows from uniqueness. 
To investigate the latter, we consider the system: 
Find $(\delta\vec X,\kappa_g, \vec Y_g) \in \xspace^h \times V^h \times 
\yspace^h$ such that
\begin{subequations}
\begin{align} 
\left(g(\vec X^m)\,\delta\vec X\,.\,\vec\omega^m, \vec\chi\,.\,\vec\omega^m\,
|\vec X^m_\rho|_g \right)^\diamond
- \ttau_m \left((\vec Y_g)_s, \vec\chi_s\,|\vec X^m_\rho|_g \right)^\diamond
& = 0 \quad \forall\ \vec\chi \in \xspace^h\,, \label{eq:proofga} \\
\left(\kappa_g - g^\frac12(\vec X^m)\,\vec Y_g\,.\,\vec\nu^m, 
\chi\,|\vec X^m_\rho|_g \right)^\diamond & = 0
\quad \forall\ \chi \in V^h\,, \label{eq:proofgb} \\
\left(g^\frac12(\vec X^m)\,\kappa_g\,\vec\nu^m, 
\vec\eta\,|\vec X^m_\rho|_g \right)^\diamond
+ \left((\delta\vec X)_s,\vec\eta_s\,|\vec X^m_\rho|_g\right)^\diamond
& = 0
\quad \forall\ \vec\eta \in \yspace^h\,. \label{eq:proofgc}
\end{align}
\end{subequations}
Choosing $\vec\chi = \delta\vec X\in\xspace^h$ in (\ref{eq:proofga}),
$\chi = \kappa_g$ in (\ref{eq:proofgb}) and $\vec\eta = \vec Y_g \in \yspace^h$
in (\ref{eq:proofgc}) yields that 
\begin{equation} \label{eq:proofgsum}
\left(g(\vec X^m)\,(\delta\vec X\,.\,\vec\omega^m)^2, 
|\vec X^m_\rho|_g \right)^\diamond
+ \ttau_m \left((\kappa_g)^2,|\vec X^m_\rho|_g \right)^\diamond = 0\,.
\end{equation}
First of all it follows from \eqref{eq:proofgsum}, 
our assumption on (\ref{eq:Idiamond}) 
and the positivities of $g(\vec X^m)$ and $|\vec X^m_\rho|_g$ 
on $\overline I\setminus \partial_0 I$, that
$\kappa_g = 0 \in V^h$.
As a consequence, we obtain by choosing $\vec\eta = \delta\vec X
\in \xspace^h \subset \yspace^h$ in (\ref{eq:proofgc}) that
$\delta\vec X$ is a constant vector. Now \eqref{eq:proofgsum} implies 
that this constant is such that
$\delta\vec X\,.\,\vec\omega^m(q_j) = 0$ for
all $q_j \in \overline I\setminus \partial_0 I$, and so
the assumption $(\mathfrak A)^h$ yields that $\delta\vec X = \vec 0$.

It remains to show that $\vec Y_g = \vec 0$.
If $\partial_C I = \emptyset$, then
we can choose $\vec\chi = \vec Y_g \in \yspace^h \subset 
\xspace^h$ in \eqref{eq:proofga}
to obtain that $\vec Y_g$ is constant in $\overline I$.
Combining \eqref{eq:proofgb} with assumption $(\mathfrak B)^\diamond$
then gives that $\vec Y_g = \vec 0$. 
If $\partial_C I \not= \emptyset$, on the other hand,
then assumption $(\mathfrak C)^\diamond$ directly gives that 
$\vec Y_g=\vec 0$, in view of \eqref{eq:proofga} and \eqref{eq:proofgb}. 
Hence there exists a unique solution
$(\vec X^{m+1}, \kappa^{m+1}_g, \vec Y^{m+1}_g) \in \Vh \times V^h \times
\yspace^h$ to $(\BGNpwfwf_m)^\diamond$.

For the case $(\cdot,\cdot)^\diamond = (\cdot,\cdot)^h$ we 
note that $V^h$ in \eqref{eq:B8b} can be equivalently replaced by 
$W^h_{\partial_0}$. Existence of a unique 
$(\vec X^{m+1}, \kappa^{m+1}_g, \vec Y^{m+1}_g) \in \Vh \times W^h_{\partial_0}
\times\yspace^h$ 
to this new system can then be shown similarly to the above proof,
which gives all the remaining desired results.
\end{proof}

\begin{rem} \label{rem:homotopic}
We note that in the examples \eqref{eq:gcatenoid}, \eqref{eq:gtorus} and
\eqref{eq:gcone}, any closed curve 
$\vec x(I)$ in $H$ will correspond to a curve $\vec\Phi(\vec x(I))$ on the
hypersurface $\mathcal{M}$ that is homotopic to a point. 
In order to model other curves, the domain $H$ needs to be embedded in an
algebraic structure different to $\bR^2$. In particular, 
$H = \bR\times \RpiZ$ for \eqref{eq:gcatenoid} and \eqref{eq:gcone},
and $H = \RpisZ \times \RpiZ$ for \eqref{eq:gtorus}, respectively. 

For the implementation of the presented schemes, this only affects the
calculation of differences of vectors in $H$. For example, for each interval
$\vec X^m(I_j)$ some care needs to be taken when selecting representatives of
the endpoints for the computation of $\vec X^m_\rho$, which then naturally
yields $|\vec X^m_\rho|$ and $\vec\nu^m$.
We will present some numerical simulations for closed curves that
are not homotopic to a point in Section~\ref{sec:nr}.
\end{rem}

\setcounter{equation}{0}
\section{Numerical results} \label{sec:nr}

We used the finite element toolbox Alberta, \cite{Alberta},
to implement our schemes. The arising linear systems are solved with the
sparse factorisation package UMFPACK, see \cite{Davis04}. Solutions to the 
nonlinear equations for the scheme $(\BGNmc_{m,\star})^{h}$
are computed with a Newton iteration.
The two schemes $(\BGNmckappa_{m})^{h}$ and $(\BGNmc_{m,\star})^{h}$ for
curvature flow produce very similar results, and can be used interchangeably.
We include numerical results for both, in order to demonstrate that they work
well in practice. However, for evolutions where numerical stability is crucial,
we often prefer to employ the unconditionally stable scheme 
$(\BGNmc_{m,\star})^{h}$.

We note from \eqref{eq:Lgh} and \eqref{eq:Lg} 
that $L^h_g(\vec X^m)$ acts as a discrete energy for 
$(\BGNmckappa_{m})^{h}$ and $(\BGNmc_{m,\star})^{h}$, while
on recalling Theorem~\ref{thm:stabg} we define $\widetilde W_{g}^{m+1} =
\tfrac12\, ( (\kappa_g^{m+1})^2,|\vec X^m_\rho|_g )^\diamond$
as a discrete analogue of (\ref{eq:Wg})
for the scheme $(\BGNpwfwf_m)^\diamond$.
As the quadrature rule for the scheme $(\BGNpwfwf_m)^\diamond$ we either 
consider (\ref{eq:ip0}), leading to $(\BGNpwfwf_m)^h$, or a quadrature that 
is exact for polynomials of degree up to five, denoted by 
$(\cdot,\cdot)^\star$, and so leading to $(\BGNpwfwf_m)^\star$.
In order to avoid the non-uniqueness issue in Lemma~\ref{lem:exg},
we always use the latter in the case of open curves.

The initial data for the scheme $(\BGNpwfwf_m)^\diamond$,
given $\Gamma^0 = \vec X^0(\overline I)$, is defined as follows.
First we define $\kappa^0 \in V^h$ via
$\kappa^0(q_j) = [|\vec\omega^0|^{-2}\,\vec\kappa^0\,.\,\vec\omega^0](q_j)$
for $j=0,\ldots,J$, 
where $\vec\kappa^0\in \Vh$ is such that
\begin{equation*} 
\left( \vec\kappa^{0},\vec\eta\, |\vec X^0_\rho| \right)^h
+ \left( \vec{X}^{0}_\rho , \vec\eta_\rho\,|\vec X^0_\rho|^{-1} \right)
 = 0 \quad \forall\ \vec\eta \in \Vh\,.
\end{equation*}
Then let $\kappa_g^0 \in W^h_{\partial_0}$ with
$\kappa_g^0(q_j) = \mathcal{K}(\kappa^0, \vec\omega^0, \vec X^0)(q_j)$ 
for $q_j \in \overline I \setminus \partial_0 I$, recall \eqref{eq:mathcalK}. 
In addition, let $\vec Y_g^0 \in [W^h_{\partial_0}]^2$ with
$\vec Y_g^0 = [g^{-\frac12}(\vec X^0)\,|\vec\omega^0|^{-2}\,
\kappa_g^0\,\vec\omega^0](q_j)$ for 
$q_j \in \overline I \setminus \partial_0 I$.

In most of the presented simulations
we use uniform time steps, $\ttau_m = \ttau$, $m=0,\ldots,M-1$.
For some simulations, however, we use an adaptive time step strategy 
satisfying $\ttau_{\min} \leq \ttau_m \leq \ttau_{\max}$, $m=0,\ldots,M-1$,
with smaller time steps at the beginning of the evolution.
Unless otherwise stated, in all the simulations we use the
discretisation parameters $J=256$ and uniform time steps $\ttau = 10^{-4}$.

\subsection{The metric \eqref{eq:gmu}}

For the scheme $(\BGNmckappa_m)^h$
we show the evolution of two cigar shapes in Figure~\ref{fig:cfmu_cigar}
for the metric \eqref{eq:gmu} with $\mu=1$.
We note that in both cases the curve shrinks to a point.
\begin{figure}
\centering
\includegraphics[angle=-90,width=0.6\textwidth]{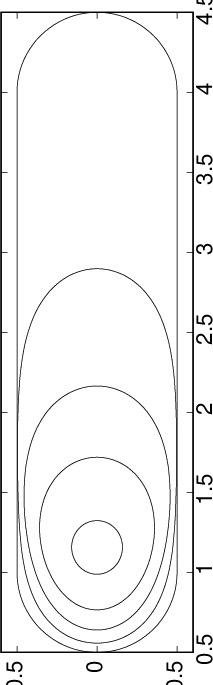} \quad
\includegraphics[angle=-90,width=0.2\textwidth]{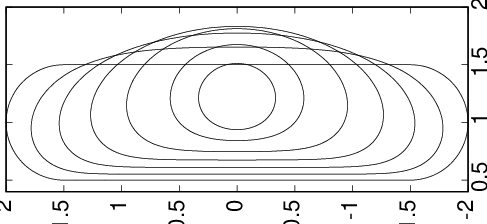}
\caption{$(\BGNmckappa_m)^h$
Curvature flow towards extinction for \eqref{eq:gmu} with $\mu=1$.
Solution at times $t=0,0.05,\ldots,0.2$ (left), and at times
$t=0,0.1,\ldots,0.5,0.55$ (right).} 
\label{fig:cfmu_cigar}
\end{figure}%
Repeating the same evolutions for  the metric
\eqref{eq:gmu} with $\mu=-1$,
now using the scheme $(\BGNmc_{m,\star})^{h}$, leads to the results
shown in Figure~\ref{fig:cfmu-1_cigar}. While the horizontally aligned curve
again shrink to a point, the vertically aligned curve approaches the
$x_2$--axis in order to minimise its geodesic length. The degeneracy of
$g$ on the axis leads to a breakdown of the evolution. In practice this means
that the Newton iteration to find a solution for $(\BGNmc_{m,\star})^{h}$
no longer converges. Here we note that we used the smaller uniform time step 
size $\ttau=10^{-5}$ for this experiment.
\begin{figure}
\centering
\includegraphics[angle=-90,width=0.6\textwidth]{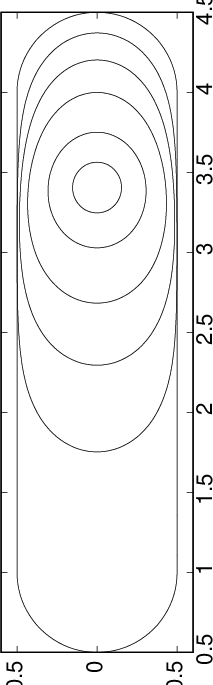}
\includegraphics[angle=-90,width=0.2\textwidth]{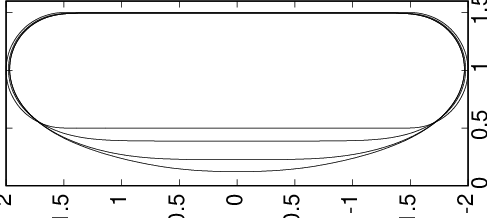}
\caption{$(\BGNmc_{m,\star})^{h}$
Curvature flow towards extinction for \eqref{eq:gmu} with $\mu=-1$.
Solution at times $t=0,1,\ldots,4,4.5$ (left) and at times
$t=0,0.01,0.015,0.0156$ (right).} 
\label{fig:cfmu-1_cigar}
\end{figure}%

We stress that the evolution is well defined, however, if we assign boundary
points to lie on the $x_2$--axis and to move freely on it. This is not 
dissimilar to the modelling of mean curvature flow for axisymmetric surfaces
of genus zero, see \cite{aximcf} for details. As an example,
we show the evolution of a semicircle with radius 1 and 
$\partial_0 I = \partial I$ in Figure~\ref{fig:cfmu-1open}.
As a comparison, we also show the same evolution for the case
$\partial_1 I = \partial I$. In both cases, the semicircle shrinks to
extinction, but the shape and time scale of the two evolutions differ.
\begin{figure}
\centering
\includegraphics[angle=-90,width=0.2\textwidth]{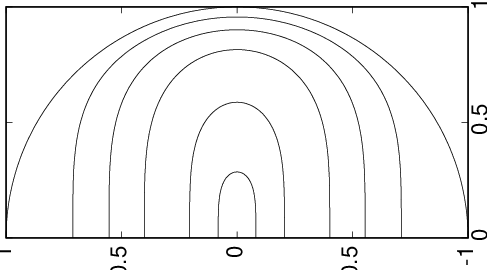}
\quad
\includegraphics[angle=-90,width=0.2\textwidth]{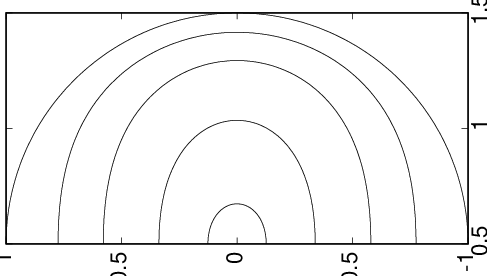}
\caption{$(\BGNmc_{m,\star})^{h}$
Curvature flow for \eqref{eq:gmu} with $\mu=-1$.
Solution for $\partial I = \partial_0 I = \{0,1\}$ at times 
$t=0,0.02,\ldots,0.08,0.085$ (left) and for 
$\partial I = \partial_1 I = \{0,1\}$ at times $t=0,0.1,\ldots,0.3,0.34$ 
(right).}
\label{fig:cfmu-1open}
\end{figure}

For completeness, we also show some evolutions for the cases
$\partial_D I = \partial I$ and $\partial_2 I = \partial I$ in
Figure~\ref{fig:cfmu-1_halfcircle001}. 
The first evolution for the Dirichlet, or no-slip, boundary conditions
leads to the curve trying to reach the $x_2$--axis in order to reduce its
length. Similarly to Figure~\ref{fig:cfmu-1_cigar} this leads to a breakdown of
the scheme. The second evolution for the Dirichlet conditions yields a straight
line segment as geodesic, while for the free-slip condition the initial
semicircle shrinks to a point on the $x_1$--axis.
\begin{figure}
\centering
\mbox{
\includegraphics[angle=-90,width=0.2\textwidth]{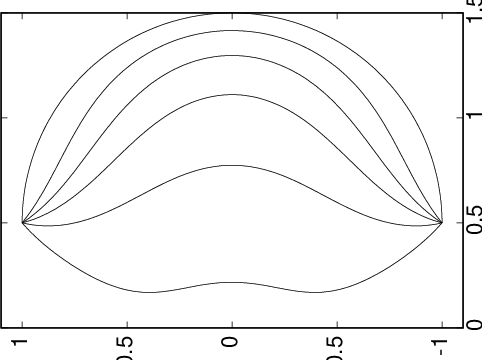}
\includegraphics[angle=-90,width=0.4\textwidth]{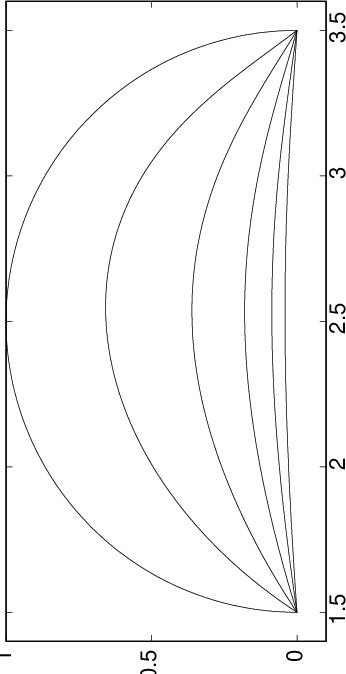}
\includegraphics[angle=-90,width=0.4\textwidth]{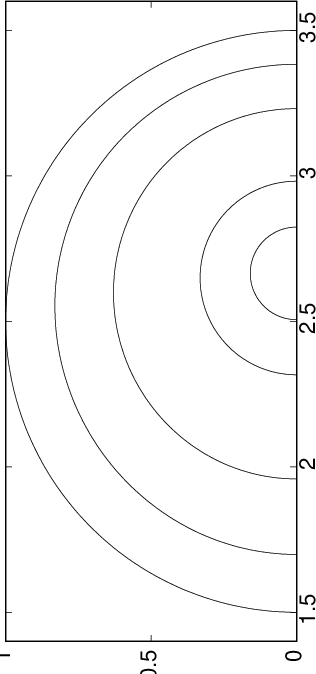}
}
\mbox{
\includegraphics[angle=-90,width=0.3\textwidth]{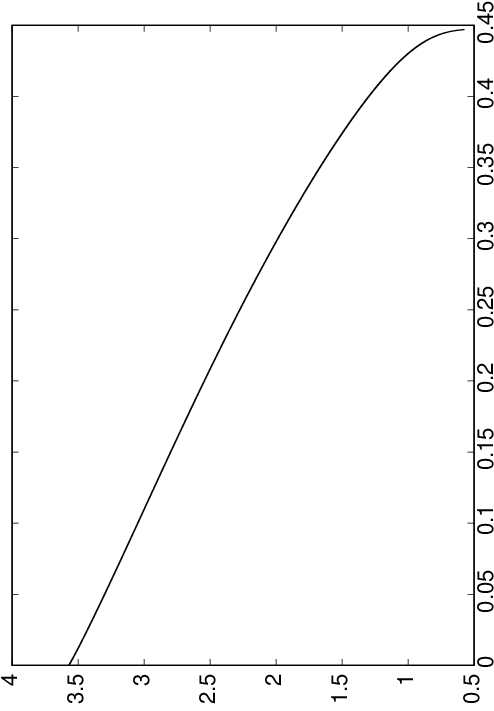}
\includegraphics[angle=-90,width=0.3\textwidth]{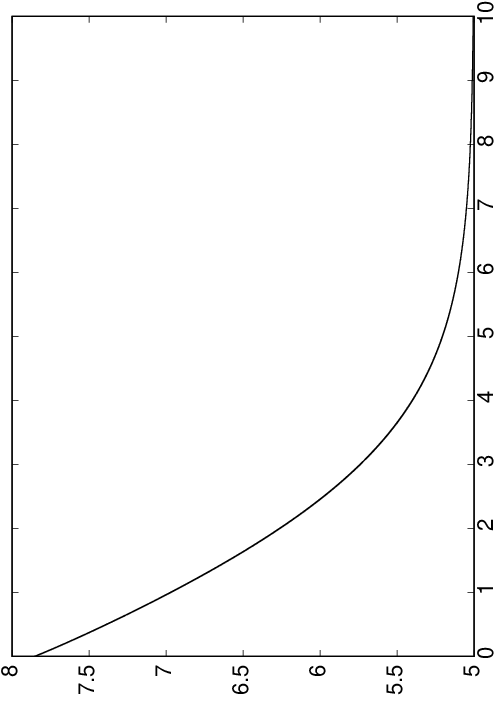}
\includegraphics[angle=-90,width=0.3\textwidth]{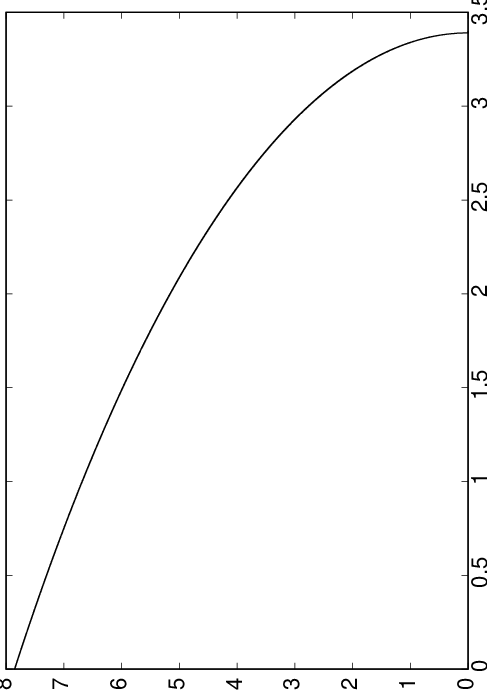}
}
\caption{$(\BGNmc_{m,\star})^{h}$
Curvature flow for \eqref{eq:gmu} with $\mu=-1$ with
$\partial I = \partial_D I = \{0,1\}$ (left and middle)
and $\partial I = \partial_2 I = \{0,1\}$ (right).
Solution at times $t=0,0.1,\ldots,0.4,0.447$ (left), 
$t=0,2,\ldots,10$ (middle) and $t=0,1,\ldots,3,3.3$ (right).
Below we show plots of the discrete energies $L_g^h(\vec X^m)$.
} 
\label{fig:cfmu-1_halfcircle001}
\end{figure}%

Evolutions for elastic flow with Navier and clamped boundary conditions,
respectively, are shown in Figure~\ref{fig:pwfwfmu-1_halfcircleDN}.
Here, for the clamped boundary conditions, recall \eqref{eq:veczeta}, 
we choose $\vec\zeta(p) = (\sin \vartheta(p), \cos \vartheta(p))^T$,
with $\vartheta(0) = 210^\circ$ and $\vartheta(1) = 150^\circ$. 
While in the Navier case the curve appears to grow
unboundedly, in the clamped case the curve seems to approach an optimal shape
aligned with the chosen metric.
\begin{figure}
\centering
\includegraphics[angle=-90,width=0.3\textwidth]{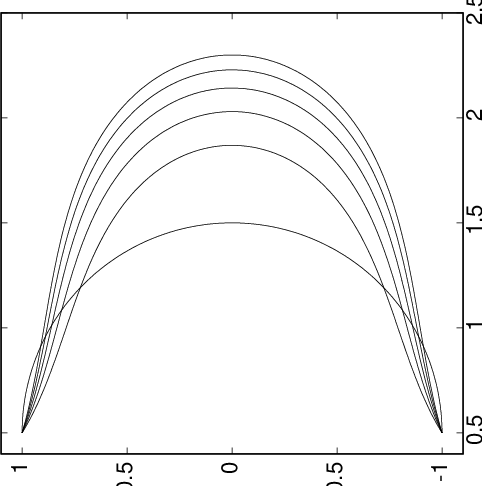}
\includegraphics[angle=-90,width=0.4\textwidth]{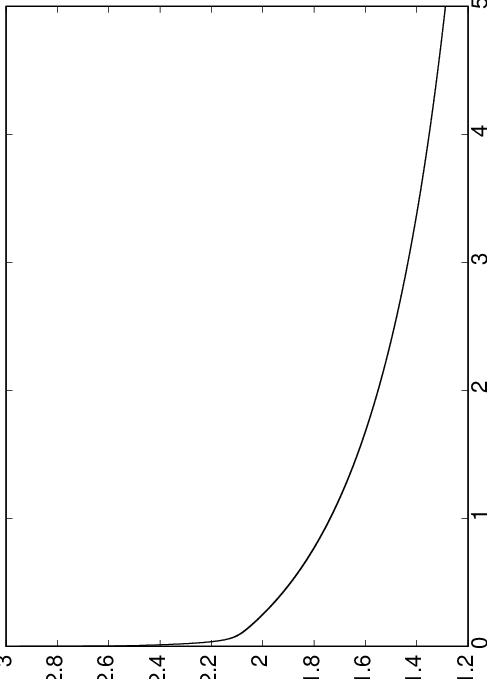}
\includegraphics[angle=-90,width=0.4\textwidth]{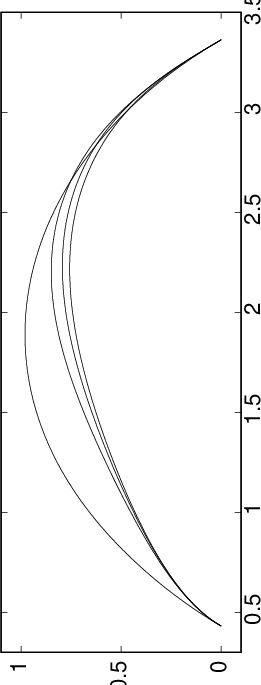}
\includegraphics[angle=-90,width=0.4\textwidth]{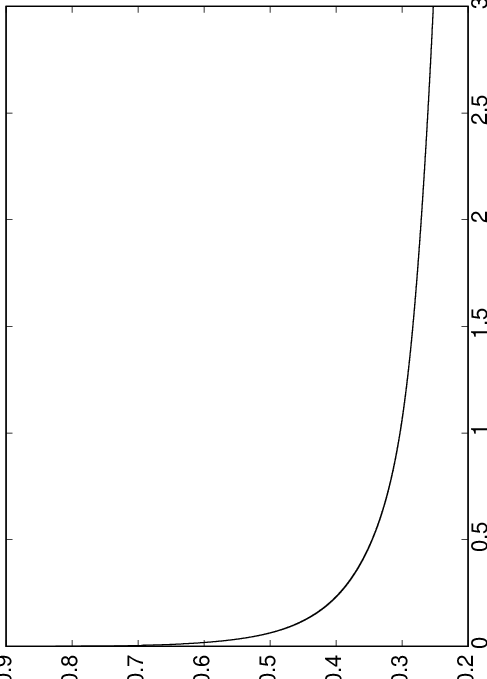}
\caption{$(\BGNpwfwf_m)^{\star}$
Elastic flow for \eqref{eq:gmu} with $\mu=-1$ and
$\partial I = \partial_N I = \{0,1\}$ (top) and
$\partial I = \partial_D I = \{0,1\}$ (bottom).
Solution at times $t=0,1,\ldots,5$ (above) and at times
$t=0,1,\ldots,5$ (below).
We also show plots of the discrete energy $\widetilde W_g^{m+1}$ over time.
} 
\label{fig:pwfwfmu-1_halfcircleDN}
\end{figure}%

We remind the reader that many more numerical simulations for closed curves
moving under curvature flow or elastic flow in the Riemannian manifold defined
by \eqref{eq:gmu}, including for the case case $\mu=1$ for the hyperbolic
plane, can be found in \cite{hypbol,hypbolpwf}.

\subsection{The torus metric \eqref{eq:gtorus}} 
A geodesic between two fixed points on the Clifford torus is computed in
Figure~\ref{fig:mctorusline}.
To this end, we employ the metric induced by
(\ref{eq:gtorus}) with $\mathfrak s = 1$, so that the torus has radii $r=1$ and
$R = \sqrt{2}$. We observe that the evolution eventually settles on a geodesic,
that is clearly not the shortest path connected the two points on the torus.
That is because of a topological restriction stemming from the fact that the
curve must stay within the equivalence class that is prescribed by the initial 
data.
\begin{figure}
\centering
\includegraphics[angle=-90,totalheight=3cm,align=t]{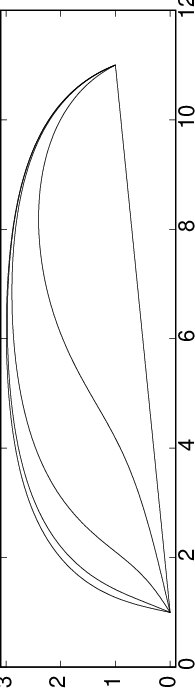}\quad
\includegraphics[angle=-0,totalheight=5cm,align=t]{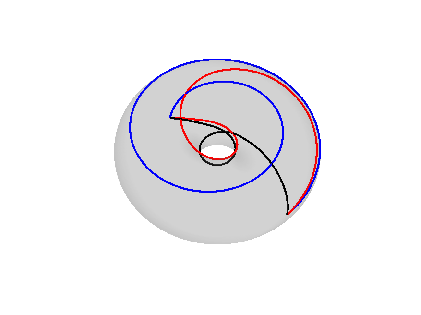}
\includegraphics[angle=-90,totalheight=4cm,align=t]{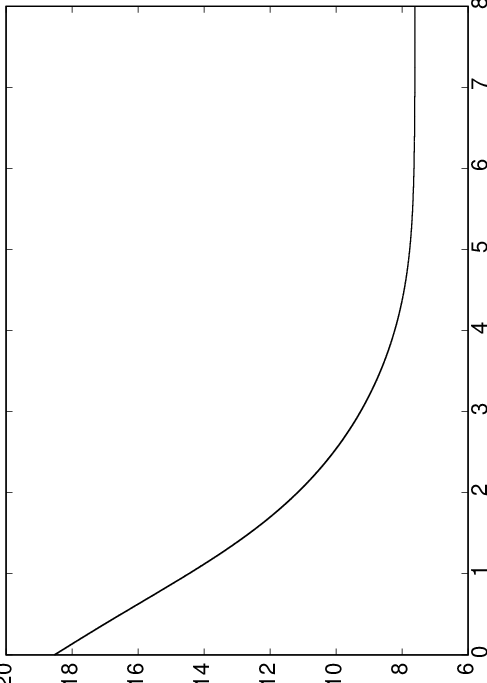}
\caption{$(\BGNmckappa_{m})^h$
Geodesic curvature flow on a Clifford torus, with $\partial_D I = \partial I =
\{0,1\}$.
The solutions $\vec X^m$ at times $t = 0, 2,\ldots, 6$. 
Below we visualise $\vec\Phi(\vec X^m)$ at times $t=0$ (blue), $t=2$ (red)
and $t=6$ (black), 
for (\ref{eq:gtorus}) with $\mathfrak s=1$, and also show a plot of the 
discrete energy $L_g^h(\vec X^m)$.
} 
\label{fig:mctorusline}
\end{figure}%

On recalling Remark~\ref{rem:homotopic}, we also present an evolution for
geodesic curvature flow of a closed curve that is not homotopic to a point. 
See Figure~\ref{fig:app:mctorus} for a presentation of the numerical
results for the scheme $(\BGNmc_{m,\star})^h$.
\begin{figure}
\centering
\mbox{
\includegraphics[angle=-0,totalheight=4cm,align=t]{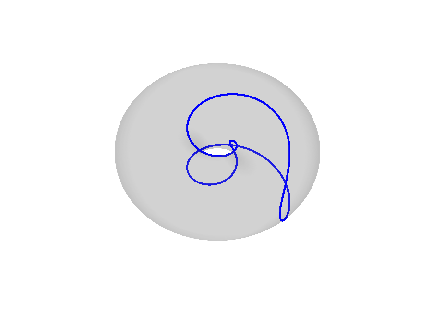}
\includegraphics[angle=-0,totalheight=4cm,align=t]{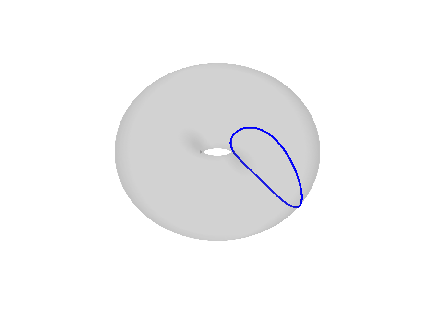}
\includegraphics[angle=-0,totalheight=4cm,align=t]{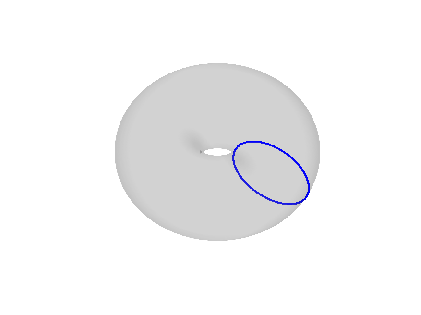}
}
\includegraphics[angle=-90,totalheight=4cm,align=t]{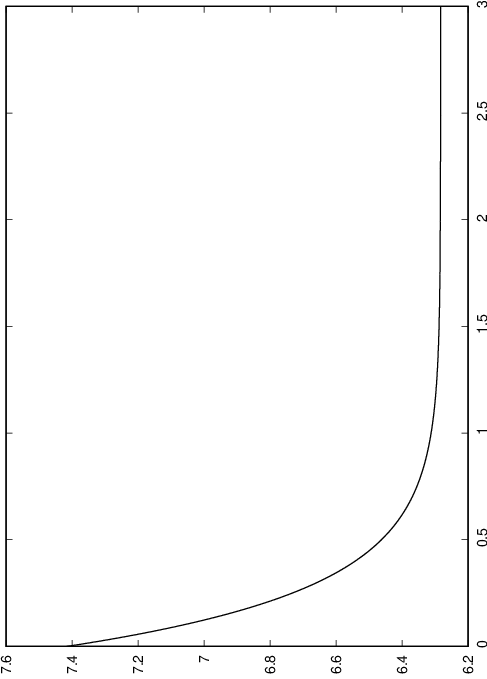}
\caption{
$(\BGNmc_{m,\star})^h$
Geodesic curvature flow on a Clifford torus.
We visualise $\vec\Phi(\vec X^m)$ at times $t=0,1,3$, 
for (\ref{eq:gtorus}) with $\mathfrak s=1$.
A plot of the discrete energy $L_g^h(\vec X^m)$ below.
} 
\label{fig:app:mctorus}
\end{figure}%

\subsection{The Angenent metric \eqref{eq:gAngenent}}

Unless otherwise stated, we choose $n=2$ in \eqref{eq:gAngenent}.
First we show the evolution under curvature flow of an eongated cigar shape 
that shrinks to a point, see in Figure~\ref{fig:cfang_cigar}.
\begin{figure}
\centering
\includegraphics[angle=-90,width=0.45\textwidth]{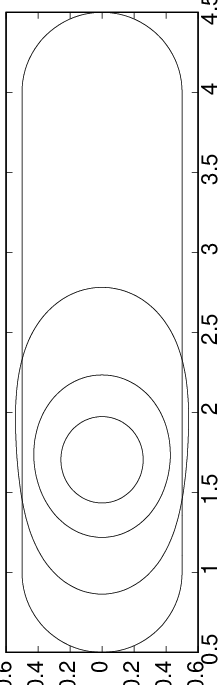}
\includegraphics[angle=-90,width=0.45\textwidth]{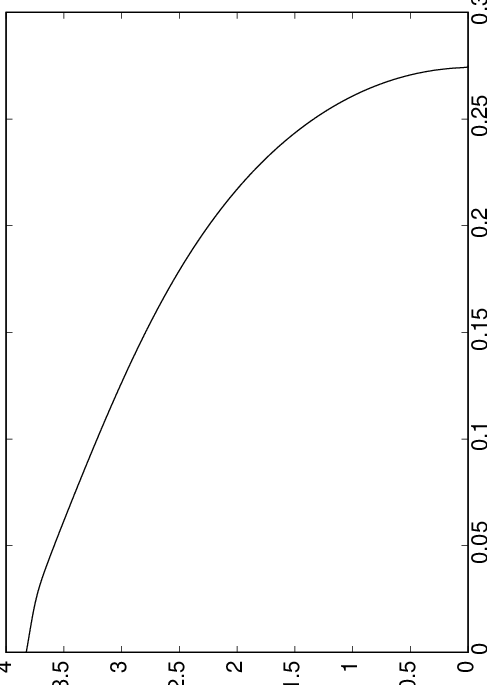}
\caption{$(\BGNmc_{m,\star})^{h}$
Curvature flow towards extinction for \eqref{eq:gAngenent}.
Solution at times $t=0,0.1,0.2,0.25$.
} 
\label{fig:cfang_cigar}
\end{figure}%

In a second experiment, we show the evolution under elastic flow 
of a circle towards the generating curve of the Angenent torus in an 
axisymmetric setting. We recall that the Angenent torus is a 
critical point of Huisken's $F$-functional \eqref{eq:HuiskenF}, 
and hence a self-shrinker for
mean curvature flow in $\bR^3$, with extinction time 1. As a consequence,
the generating curve of the Angenent torus, which from now on we will also 
simply call Angenent torus, is a critical point of the geodesic length $L_g$, 
and hence a geodesic. 
For the evolution shown in Figure~\ref{fig:gangenentwf},
we observe that the discrete curvature energy $W_g^{m+1}$ reduces from
about $3.5$ to about $10^{-5}$, giving a strong indication that
we have indeed found a geodesic. 
Note also that the final shape in Figure~\ref{fig:gangenentwf}
agrees well with the numerical results in 
\cite{Chopp94,Berchenko-Kogan19,schemeD}.
\begin{figure}
\center
\includegraphics[angle=-90,width=0.4\textwidth]{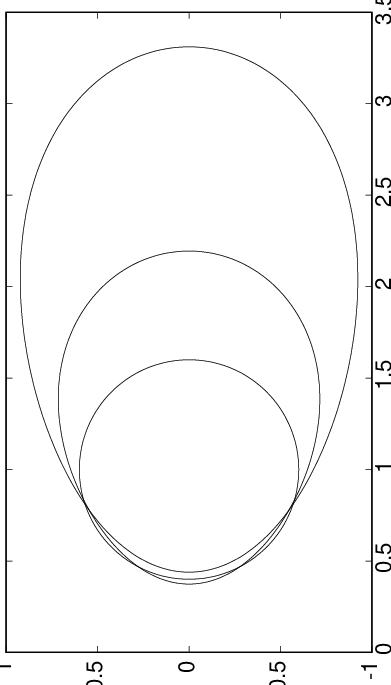}
\includegraphics[angle=-90,width=0.4\textwidth]{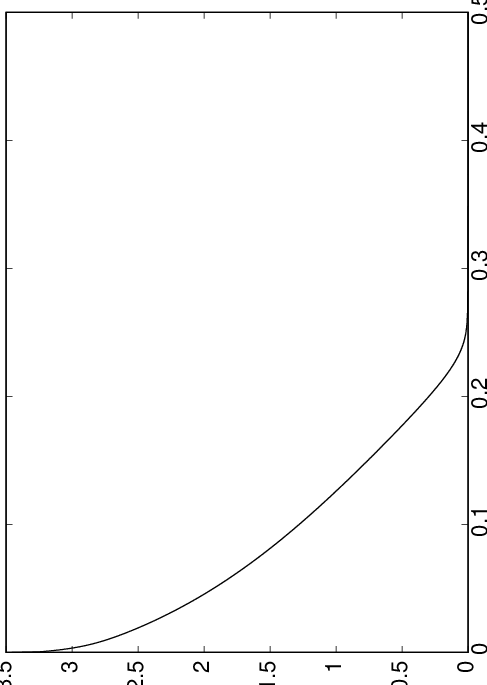}
\caption{$(\BGNpwfwf_{m})^\star$
Elastic flow for \eqref{eq:gAngenent} towards the Angenent torus.
Plots are at times $t=0,0.1,0.5$.
We also show a plot of the discrete energy $\widetilde W_g^{m+1}$ over time.
}
\label{fig:gangenentwf}
\end{figure}%
We have also performed simulations for elastic flow of initial curves
with a winding number larger than one, with respect to the 
point $2\,\vec\ek_1$, and they always settle as a stationary solution on a
multiple covering of the Angenent torus.

It is known that the Angenent torus is an unstable critical point
of the geodesic length $L_g$, see 
\cite{ColdingM12,Liu16,Berchenko-Kogan20preprint},
and this is confirmed by our numerical experiments.
Hence it is practically impossible to obtain an approximation to it
as a long-time limit of curvature flow. We demonstrate this phenomenon
by starting two simulations for the stable scheme $(\BGNmc_{m,\star})^h$ 
from slightly shifted Angenent tori.
Our numerical results in Figure~\ref{fig:gangenent} confirm that the 
stationary solution is unstable, and we see the
curve either moving monotonically
towards the $x_2$--axis, or towards infinity, with a 
significant decrease in the geodesic length of the curve in each case.
For these experiments we used the finer discretisation parameters
$J=2048$ and $\ttau=10^{-5}$.
\begin{figure}
\center
\includegraphics[angle=-90,width=0.4\textwidth]{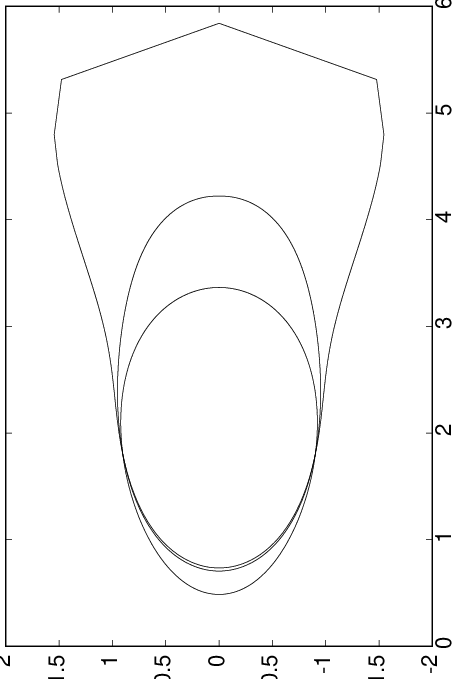}
\includegraphics[angle=-90,width=0.4\textwidth]{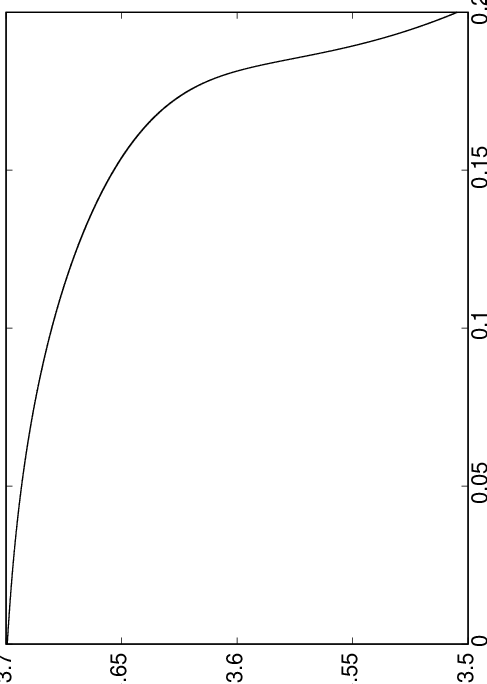}
\includegraphics[angle=-90,width=0.4\textwidth]{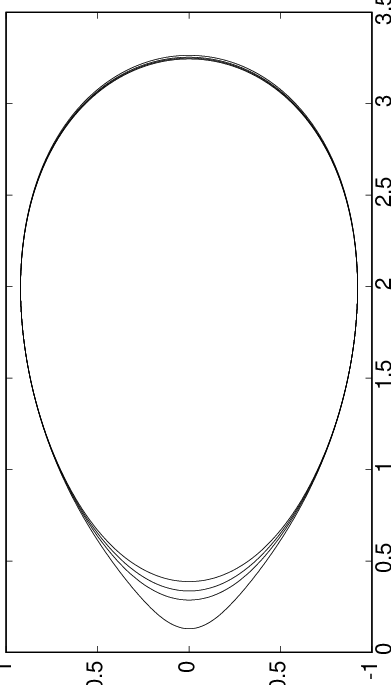}
\includegraphics[angle=-90,width=0.4\textwidth]{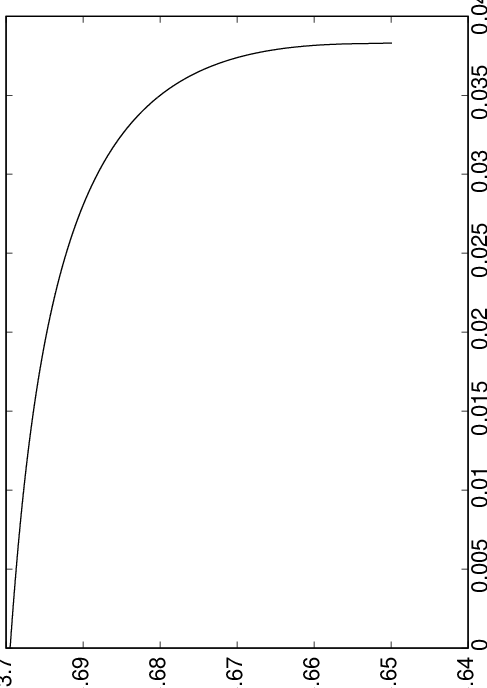}
\caption{$(\BGNmc_{m,\star})^h$
Curvature flow for \eqref{eq:gAngenent}, starting from horizontally shifted 
Angenent tori. Above shifted by $0.05$ to the right,
below shifted by $0.05$ to the left.
Plots are at times $t=0,0.18,0.2$ (above) and
at times $t=0,0.02,0.03,0.038$ (below).
We also show plots of the discrete energy $L_g^h(\vec X^m)$ over time.
}
\label{fig:gangenent}
\end{figure}%

We highlight the capabilities of our numerical method by computing
the ``Angenent tori'' in dimensions four and five, that is hypersurfaces
in $\bR^{n+1}$ that are topologically equivalent to $\bS^1 \times \bS^{n-1}$, 
$n=3,4$, and that are
self-shrinkers for mean curvature flow with extinction time 1. In particular,
in Figure~\ref{fig:gangenent34wf} we show the numerical steady states for
approximations of elastic flow for the metric \eqref{eq:gAngenent}, with
$n=2,3,4$. In each case the final discrete energy satisfies
$|\widetilde W_g^M| < 10^{-9}$, confirming that we are indeed approximating
geodesics.
\begin{figure}
\center
\includegraphics[angle=-90,width=0.3\textwidth]{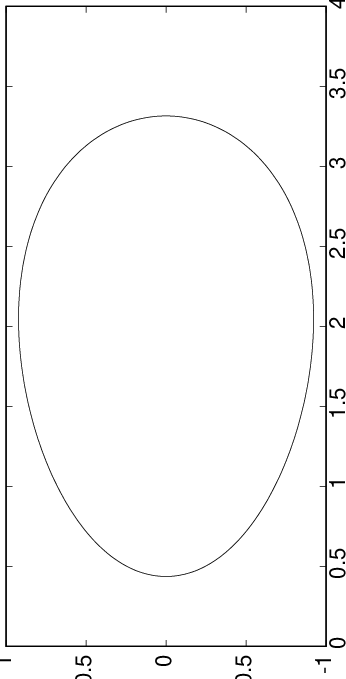}
\includegraphics[angle=-90,width=0.3\textwidth]{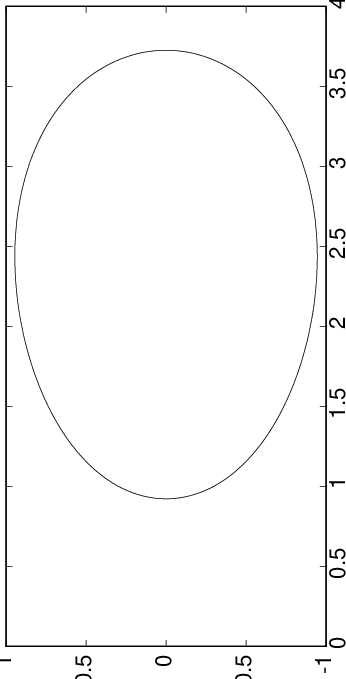}
\includegraphics[angle=-90,width=0.3\textwidth]{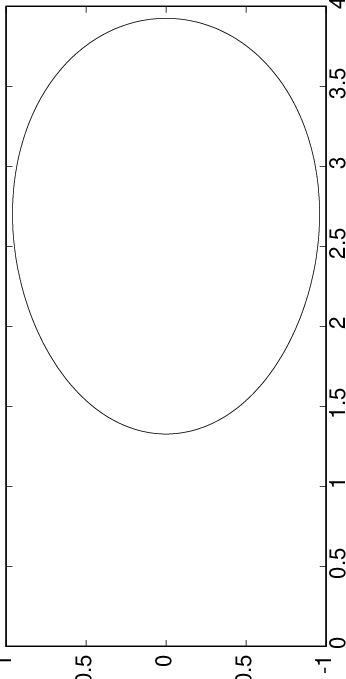}
\caption{$(\BGNpwfwf_{m})^\star$
Steady states for elastic flow for \eqref{eq:gAngenent} with $n=2,3,4$. 
Their discrete geodesic lengths are $3.70$, $6.39$ and $14.27$.
}
\label{fig:gangenent34wf}
\end{figure}%

We note that in the study of mean curvature flow self-shrinkers the value of
Huisken's $F$-functional itself is also a relevant quantity of interest, see
e.g.\ \cite{Stone94,ColdingM12,DruganN18,Berchenko-Kogan19}. For example, for
a self-shrinker its value of $F$ is equal to its entropy.
On recalling \eqref{eq:HuiskenF}, \eqref{eq:gAngenent} and
\eqref{eq:Lg}, we note that if
$\vec y(\overline I) \subset H$ is the generating curve 
for a rotationally symmetric hypersurface $\mathcal{S} \subset \bR^{n+1}$,
then
\begin{equation} \label{eq:entropy}
F(\mathcal{S}) = (4\,\pi)^{-\frac n2}\,\mathcal{H}^{n-1}(\bS^{n-1})
\left( (\vec y\,.\,\vec\ek_1)^{n-1}\,e^{-\frac14\,|\vec y|^2}, 
|\vec y_\rho| \right)
= 2^{1-n}\,[\bGamma(\tfrac{n}2)]^{-1}\,L_g(\vec y)\,.
\end{equation}
Hence in order to deduce the value of Huisken's $F$-functional for the
self-shrinkers that we have found computationally, it is sufficient to report
on the length of the corresponding geodesics, which we will do from now
on within the captions of the relevant figures. 
Here we note that the pre-factor
in \eqref{eq:entropy} for the cases $n=2,3,4$ is given by
$2^{1-n}\,[\bGamma(\tfrac{n}2)]^{-1} = \frac12$, $\frac1{2\,\sqrt{\pi}}$ 
and $\tfrac18$, respectively.
For the three geodesics displayed in Figure~\ref{fig:gangenent34wf}, our
computed values for the final discrete length $L_g^h(\vec X^M)$ are
$3.70$, $6.39$ and $14.27$, giving approximate values for the entropy of the 
associated surfaces of revolution of $1.85$, $1.80$ and $1.78$, respectively.
We remark that the value for $n=2$ agrees very well with the results
reported in \cite{Berchenko-Kogan19,schemeD}.

Denoting the entropy of the $n$-dimensional ``Angenent torus'' with 
$E_n$, we have so far established that $E_2 \approx 1.85$, $E_3 \approx 1.80$ 
and $E_4 \approx 1.78$. Continuing the above procedure for increasing
values of $n$, we numerically obtain that $E_6 \approx 1.77$,
$E_8 \approx 1.76$ and $E_{12} \approx 1.75$. This leads us to conjecture that
$E_n$ is a strictly monotonically decreasing sequence with 
$E_n \searrow \sqrt{3}$ as $n\to\infty$. 
Note that this conjecture is in the spirit of Lemma~A.4 in \cite{Stone94}. 

Of course, the most famous self-shrinker for mean curvature flow 
in $\bR^{n+1}$, with extinction time 1, is the sphere of radius $\sqrt{2\,n}$,
see e.g.\ \cite{ColdingM12}.
In the context of the metric \eqref{eq:gAngenent}, these correspond to
geodesics in the shape of semicircles with radius $\sqrt{2\,n}$. For $n=2$ and
$n=3$ we show an evolution each for elastic flow towards these geodesics, see
Figure~\ref{fig:pwfwfang_sphere23} for details,
where in each case as initial data we choose a semicircle of radius $n-1$.
For the two geodesics displayed in Figure~\ref{fig:pwfwfang_sphere23}, our
computed values for the final discrete length $L_g^h(\vec X^M)$ are
$2.94$ and $5.15$, giving approximate values for the entropy of the 
associated spheres of $1.47$ and $1.45$.
We remark that these values agree very well with the known values from
\cite[Examples~A.3]{Stone94}, which are known to converge to $\sqrt{2}$ as 
$n\to\infty$, recall \cite[Lemma~A.4]{Stone94}.
\begin{figure}
\centering
\mbox{
\includegraphics[angle=-90,width=0.2\textwidth]{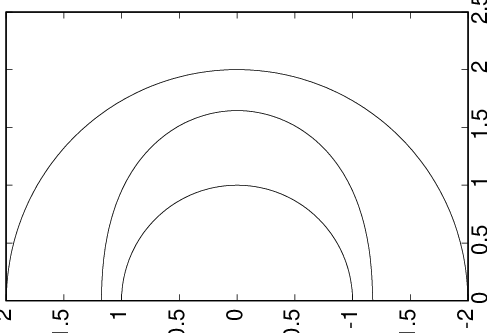}
\includegraphics[angle=-90,width=0.3\textwidth]{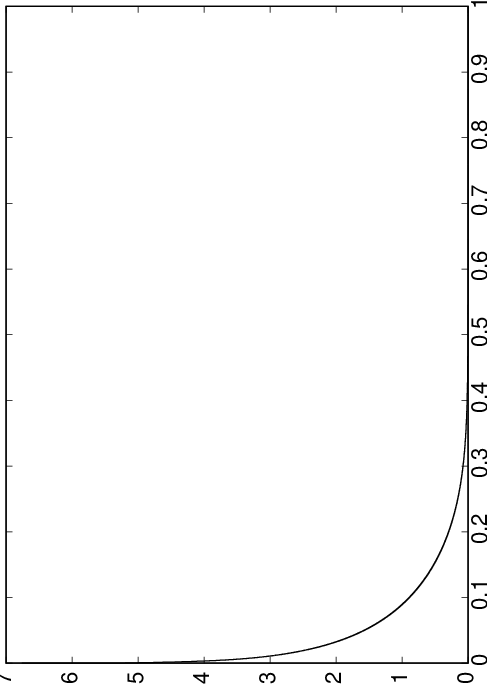}
\includegraphics[angle=-90,width=0.2\textwidth]{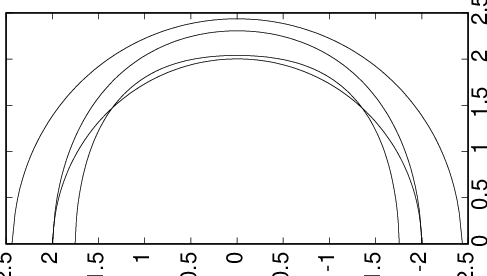}
\includegraphics[angle=-90,width=0.3\textwidth]{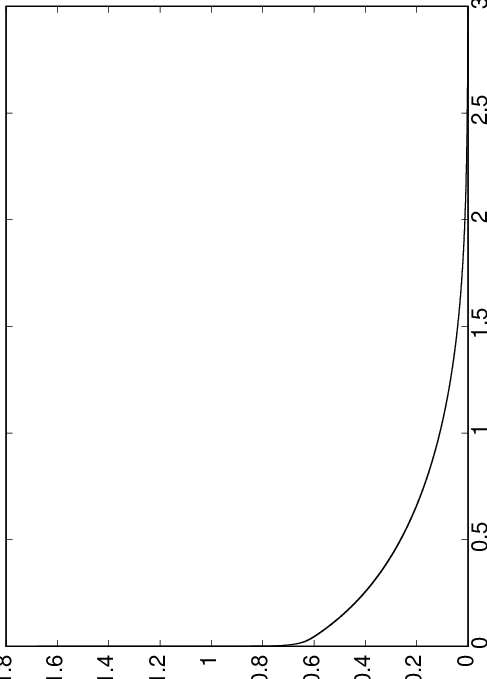}}
\caption{$(\BGNpwfwf_m)^{\star}$
Elastic flow for \eqref{eq:gAngenent}, $n=2$ (left) and $n=3$ (right), and
$\partial I = \partial_0 I = \{0,1\}$.
Solutions at times $t=0,0.1,1$ (left) and at times 
$t=0,0.1,1,3$ (right).
We also show a plot of the discrete energy $\widetilde W_g^{m+1}$ over time.
The discrete lengths of the obtained geodesics are 
$2.94$ ($n=2$) and $5.15$ ($n=3$).}
\label{fig:pwfwfang_sphere23}
\end{figure}%

In order to investigate the numerical accuracy of our proposed method, in
Table~\ref{tab:sphere2} we compare the numerical steady state solutions of 
elastic flow in the case $n=2$ with a semicircle of radius $2$,
as well as their respective induced entropy values. That is, we
list the errors $\mathcal{E}^M = \max_{j=0,\ldots,J} | 2 - |\vec X^M(q_j)| |$
and $e^M = |\frac12 L_g^h(\vec X^M) - \frac 4e|$
for increasing values of $J$. Here we fix $T=10$ and $\ttau=10^{-5}$, and
always start from the approximation of a unit semicircle.
The results presented in Table~\ref{tab:sphere2} appear to show a 
convergence rate of $1.5$ for the discrete $L^\infty$ error $\mathcal{E}^M$,
while the entropy error $e^M$ appears to converge quadratically.
\begin{table}[t!]
\center
\begin{tabular}{|r|c|c|c|c|}
\hline
$J$ & $\mathcal{E}^M$ & EOC & $e^M$ & EOC \\
\hline
32  & 1.9076e-02 & ---  & 1.4523e-03 & ---  \\
64  & 6.7016e-03 & 1.51 & 3.7588e-04 & 1.95 \\         
128 & 2.3596e-03 & 1.51 & 9.5579e-05 & 1.98 \\
256 & 8.3177e-04 & 1.50 & 2.4097e-05 & 1.99 \\
512 & 2.9555e-04 & 1.49 & 6.0496e-06 & 1.99 \\ 
\hline
\end{tabular}
\caption{Errors and experimental orders of convergence for the convergence 
test corresponding to the final shape of the evolution on the left of
Figure~\ref{fig:pwfwfang_sphere23}.}
\label{tab:sphere2}
\end{table}%

The final simulations in this subsection are devoted to finding self-shrinkers
for mean curvature flow that are non-embedded, inspired by the work
\cite{DruganK17}. We begin with an experiment for a closed curve with seven
self-intersections, see Figure~\ref{fig:DruganK17fig6}.
Under elastic flow the curve evolves towards the
generating curve of a non-embedded self-shrinker
for mean curvature flow. In fact, the steady state corresponds to the shape in
\cite[Figure~6]{DruganK17}.
Due to the large energy decrease at the beginning
of the evolution, we use an adaptive time stepping strategy with
$\ttau_{\min} = 10^{-7}$ and $\ttau_{\max} = 10^{-6}$. The spatial
discretisation uses $J=512$. The discrete energy of the final solution
satisfies $|\widetilde W^M_g| < 10^{-9}$, confirming that we are indeed
approximating a geodesic.
We note that the discrete geodesic length of the final curve is $10.53$,
giving an entropy of $5.26$.
\begin{figure}
\center
\includegraphics[angle=-90,width=0.3\textwidth]{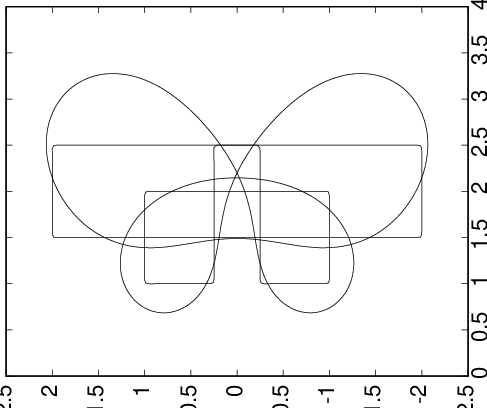}
\includegraphics[angle=-90,width=0.3\textwidth]{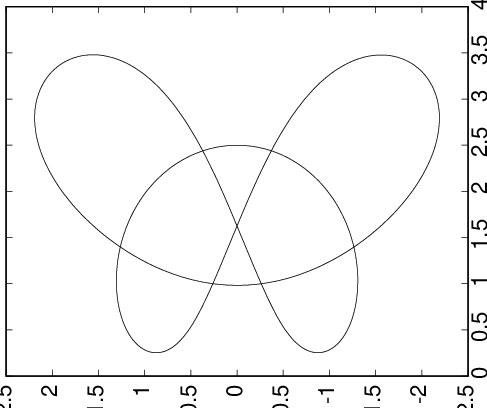}
\includegraphics[angle=-90,width=0.3\textwidth]{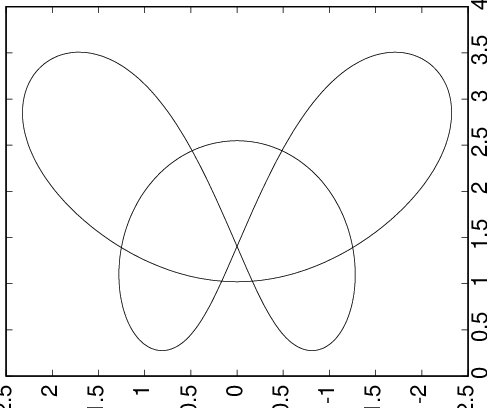}
\caption{$(\BGNpwfwf_{m})^\star$
Elastic flow for \eqref{eq:gAngenent}. 
Solution at times $t=0,0.01$, $t=1$ and $t=20$.
The discrete geodesic length of the final curve is $10.53$.
}
\label{fig:DruganK17fig6}
\end{figure}%

We also investigate, what happens to the geodesic from 
Figure~\ref{fig:DruganK17fig6} if we change the metric to \eqref{eq:gAngenent}
with $n=3$. See Figure~\ref{fig:DK17_n3} for a plot of the obtained 
numerical result, which compared to the geodesic for $n=2$ has shifted further
to the right. For this experiment we once again used an adaptive time stepping
strategy. We note that the discrete energy of the 
final solution satisfies $|\widetilde W^M_g| < 10^{-8}$.
The discrete geodesic length of the final curve is $18.24$,
giving an entropy of $5.15$.
\begin{figure}
\center
\includegraphics[angle=-90,width=0.3\textwidth]{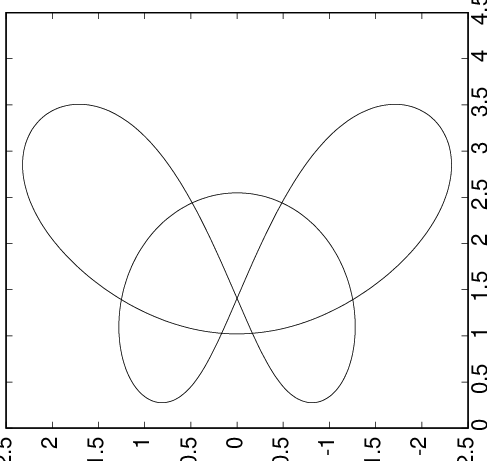}
\includegraphics[angle=-90,width=0.3\textwidth]{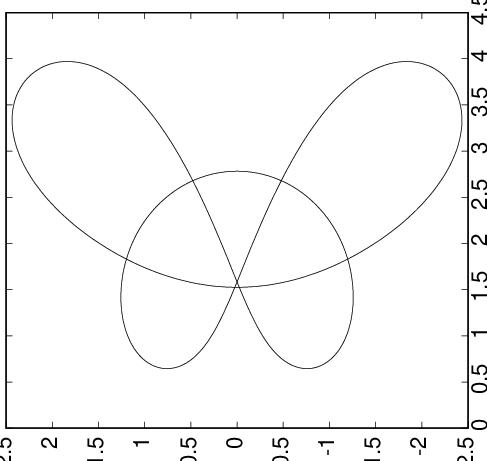}
\includegraphics[angle=-90,width=0.3\textwidth]{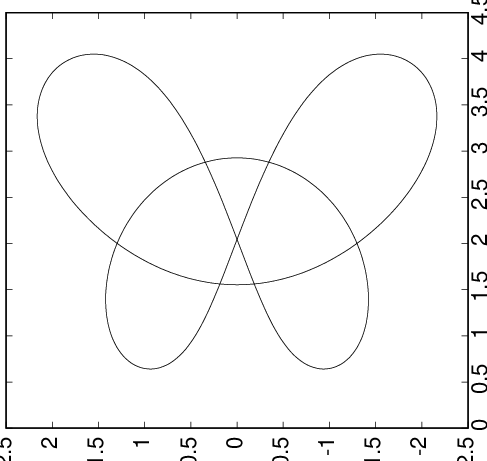}
\caption{$(\BGNpwfwf_{m})^\star$
Elastic flow for \eqref{eq:gAngenent} with $n=3$. 
Solution at times $t=0,1,100$.
The discrete geodesic length of the final curve is $18.24$.
}
\label{fig:DK17_n3}
\end{figure}%

Inspired by \cite[Figure~3]{DruganK17}, we now perform a numerical simulation
to find a non-embedded self-shrinker
of genus zero for mean curvature flow.
Starting from an initial curve with three self-intersections, we observe the
evolution for elastic flow shown in Figure~\ref{fig:DruganK17adapt},
where the final discrete energy satisfies $|\widetilde W^M_g| < 10^{-9}$.
We note the excellent agreement with \cite[Figure~3]{DruganK17}.
Here we again make use of an adaptive time stepping strategy with
$\ttau_{\min} = 10^{-7}$ and $\ttau_{\max} = 10^{-4}$. The spatial
discretisation uses $J=512$.
We note that the discrete geodesic length of the final curve is $7.33$,
giving an entropy of $3.66$.
\begin{figure}
\center
\includegraphics[angle=-90,width=0.3\textwidth]{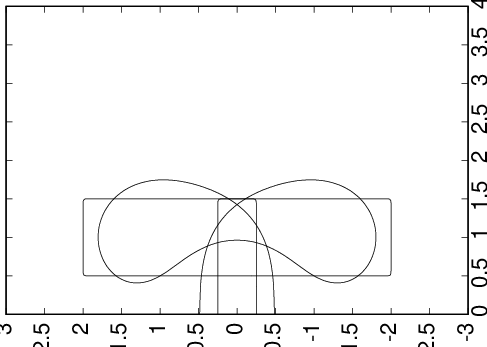}
\includegraphics[angle=-90,width=0.3\textwidth]{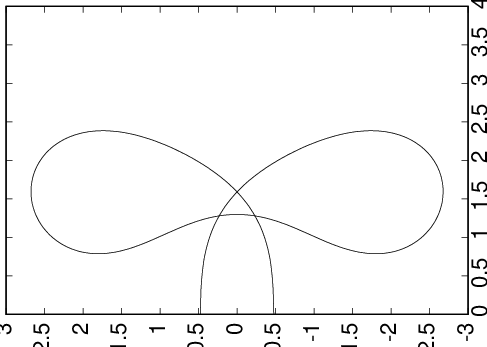}
\includegraphics[angle=-90,width=0.3\textwidth]{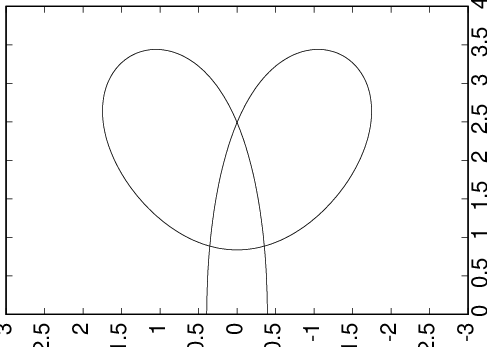}
\caption{$(\BGNpwfwf_{m})^\star$
Elastic flow for \eqref{eq:gAngenent} and
$\partial I = \partial_0 I = \{0,1\}$. 
Solution at times $t=0,0.01$, $t=0.1$ and $t=10$.
The discrete geodesic length of the final curve is $7.33$.
}
\label{fig:DruganK17adapt}
\end{figure}%

\subsection{The cone metric \eqref{eq:gcone}} 

In a first experiment for the metric \eqref{eq:gcone}, we look at (geodesic) 
curvature flow for a curve on a cone with $\mathfrak{b} = 0.5$, and so 
$\beta = 3^{-\frac12}$ in \eqref{eq:coneM}.
For the simulation in Figure~\ref{fig:mccone} it can be observed that 
in $H$ the initial circle of radius 2 deforms and shrinks to a point. 
On the hypersurface $\mathcal{M} = \vec\Phi(H)$,
the initial curve is homotopic to a point, and so shrinks to
a point away from the apex.
\begin{figure}[t!]
\centering
\includegraphics[angle=-90,width=0.3\textwidth,align=t]{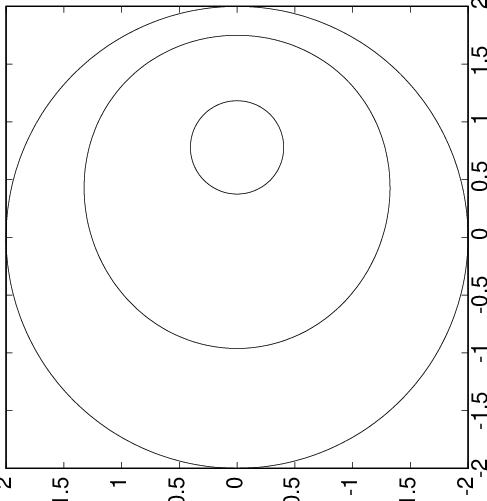}
\includegraphics[angle=-0,width=0.3\textwidth,align=t]{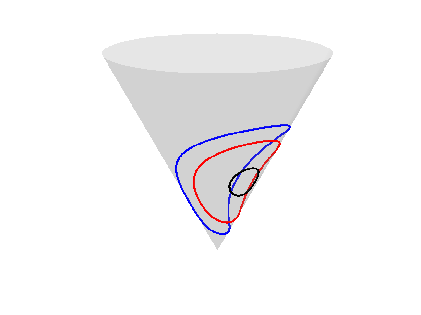}
\includegraphics[angle=-90,width=0.35\textwidth,align=t]{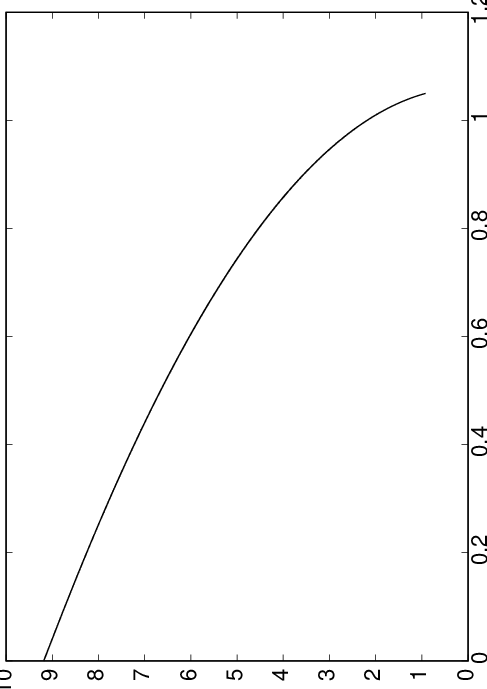}
\caption{
$(\BGNmckappa_m)^h$
Geodesic curvature flow on a cone.
The solutions $\vec X^m$
at times $t = 0, 0.5, 1$. On the right we visualise
$\vec\Phi(\vec X^m)$ at times $t=0, 0.5, 1$, for (\ref{eq:gcone}) with 
$\mathfrak b=0.5$. A plot of the discrete energy $L_g^h(\vec X^m)$ is shown on
the right.
} 
\label{fig:mccone}
\end{figure}%

The following conjecture on geodesic curvature flow on a cone was formulated by
Charles M. Elliott, \cite{Elliott09private}.

\begin{conjecture} \label{conj:cme}
A closed curve on a cone $\mathcal M$, that is not homotopic to a point on 
$\mathcal M$, will under geodesic curvature flow converge to the apex in 
finite time.
\end{conjecture}
The conjecture says, in particular, 
that the singularity of the flow will happen at the apex. 
Indeed we expect that all the points of the curve converge
to the apex at the singular time.
Moreover, we expect that a similar conjecture
holds on more general surfaces on which a curve encloses a singularity.

On recalling Remark~\ref{rem:homotopic}, we now numerically test the conjecture
by starting an evolution for geodesic curvature flow 
with a closed curve that is very close to the apex, but not uniformly so.
That is, we vary the $x_2$--coordinate
of the initial curve in $H$ between $\pm2$. During the evolution, 
the parts of the curve closest to the apex first start to rise, making the
curve becoming more circle-like, before the whole curve sinks towards to apex. 
See Figure~\ref{fig:app:mccme}, where we also show a plot of
the lowest point of the curve on the cone over time, highlighting the rise
and fall of the curve on the cone.
The observed behaviour is consistent with Conjecture~\ref{conj:cme}
\begin{figure}
\centering
\mbox{
\hspace{-2cm}
\includegraphics[angle=0,width=0.4\textwidth,align=t]{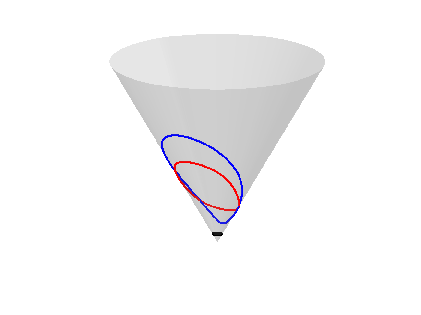}
\hspace{-1cm}
\includegraphics[angle=-90,width=0.35\textwidth,align=t]{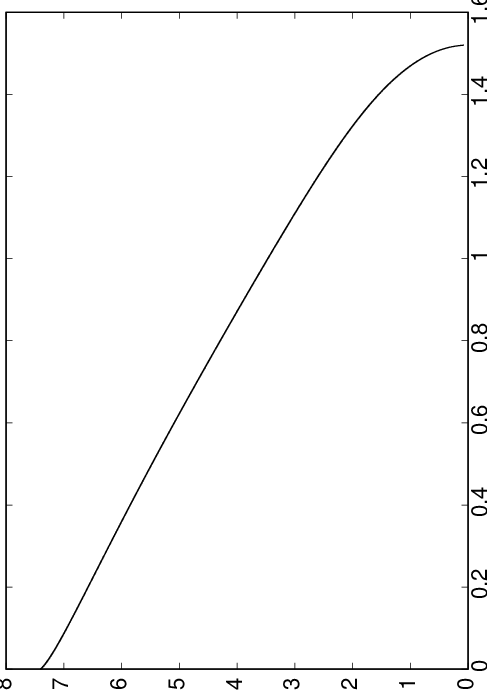}
\includegraphics[angle=-90,width=0.35\textwidth,align=t]{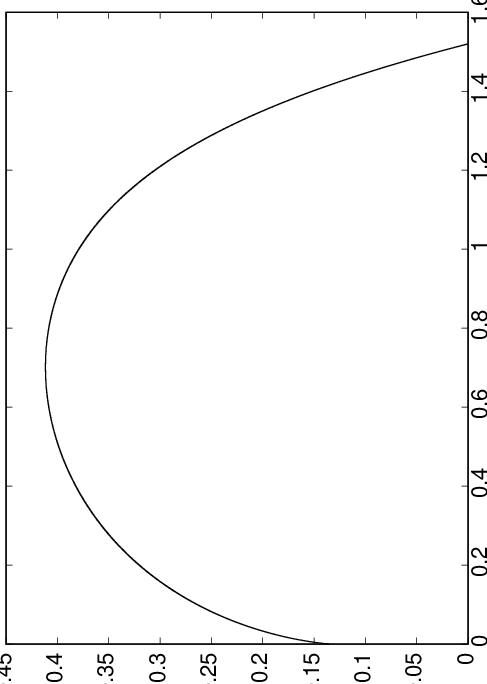}}
\caption{$(\BGNmckappa_m)^{h}$
Geodesic curvature flow on a cone.
We visualise $\vec\Phi(\vec X^m)$ at times $t=0, 0.5, 1.5$, 
for (\ref{eq:gcone}) with $\mathfrak b=0.5$. 
A plot of the discrete energy $L_g^h(\vec X^m)$ over the time interval 
$[0,1.52]$ in the middle.
On the right a plot of the lowest point of the curve on the cone, 
$\exp(\mathfrak b\,\min_{\overline I} \vec X^m\,.\,\vec\ek_1)$.
} 
\label{fig:app:mccme}
\end{figure}%

An experiment for (geodesic) elastic flow on the same cone is shown in
Figure~\ref{fig:app:elastcone}. Here the closed curve first approaches 
a circle, which then rises along the cone. 
By computing the energy one observes that a circle with increasing radius 
reduces the elastic energy.
\begin{figure}
\centering
\hspace{-2cm}
\includegraphics[angle=0,width=0.5\textwidth,align=t]{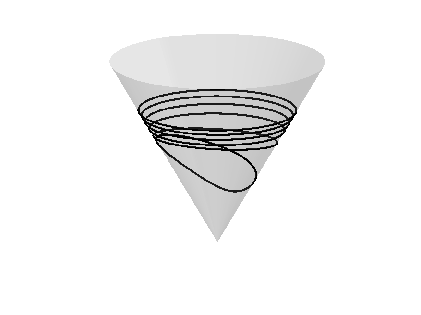}
\includegraphics[angle=-90,width=0.4\textwidth,align=t]{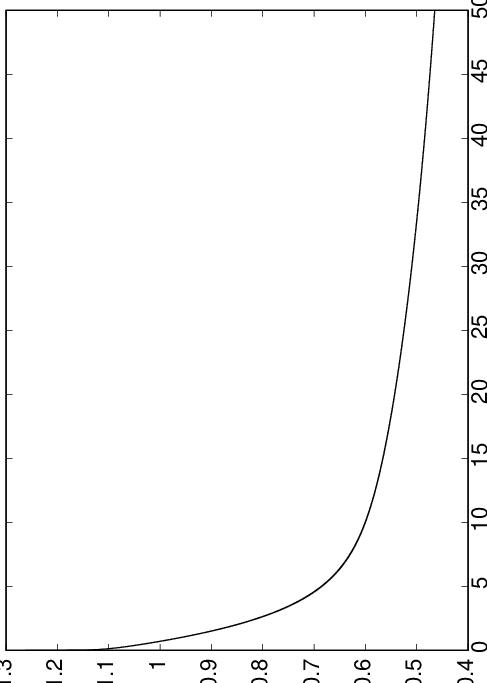}
\caption{$(\BGNpwfwf_{m,\star})^{h}$
Geodesic elastic flow on a cone.
We visualise
$\vec\Phi(\vec X^m)$ at times $t=0, 10, \ldots, 50$, for (\ref{eq:gcone}) with 
$\mathfrak b=0.5$. A plot of the discrete energy $\widetilde W_g^{m+1}$ on the 
right.
} 
\label{fig:app:elastcone}
\end{figure}%

\subsection{The metric \eqref{eq:gGNS}}

We end the section on the numerical results for our presented schemes
with some simulations for the metric \eqref{eq:gGNS} with
\eqref{eq:U} and \eqref{eq:Psi}. Recall that now geodesics in $H$ correspond 
to optimal interface profiles in multi-component phase field models. Of
particular interest are geodesics, or shortest paths, that connect
the vertices ${\bf e}_1$, ${\bf e}_2$, ${\bf e}_3$ of the Gibbs simplex G,
recall \eqref{eq:G}. To this end, we note that with the choice \eqref{eq:U}, 
it holds that the map $\vec z \mapsto f(\vec z) = {\bf u}_0 + U\,\vec z$ 
satisfies $f(0,0) = {\bf e}_1$,
$f(-2^\frac12,0) = {\bf e}_2$ and
$f(-2^\frac12,-(\frac32)^\frac12) = {\bf e}_3$.
For the first experiment we set 
$(\sigma_{12},\sigma_{13},\sigma_{23},\sigma_{123}) = (4,6,9,0)$,
and numerically compute a geodesic connecting ${\bf e}_1$ and ${\bf e}_2$
with the help of geodesic curvature flow. Here we always use the scheme
$(\BGNmckappa_{m})^{h}$ with the uniform time step size $\ttau = 10^{-5}$.
The results are shown in Figure~\ref{fig:gGNSline}, where we see that the flow
quickly settles on a curved geodesic.
We repeat the same simulation also for the paths
connecting the pure phases ${\bf e}_1$ and ${\bf e}_3$, as well as
${\bf e}_2$ and ${\bf e}_3$, and plot all three solutions within the Gibbs
simplex $G$, recall \eqref{eq:G}, at the bottom of Figure~\ref{fig:gGNSline}.
\begin{figure}
\centering
\includegraphics[angle=-90,width=0.4\textwidth,align=t]{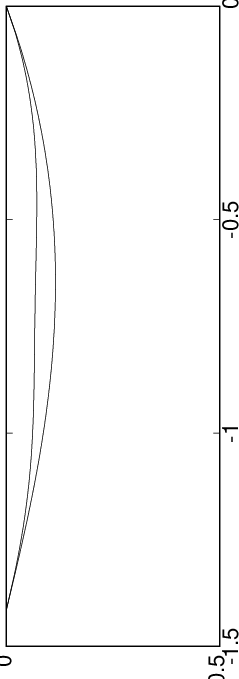}
\quad
\includegraphics[angle=-90,width=0.4\textwidth,align=t]{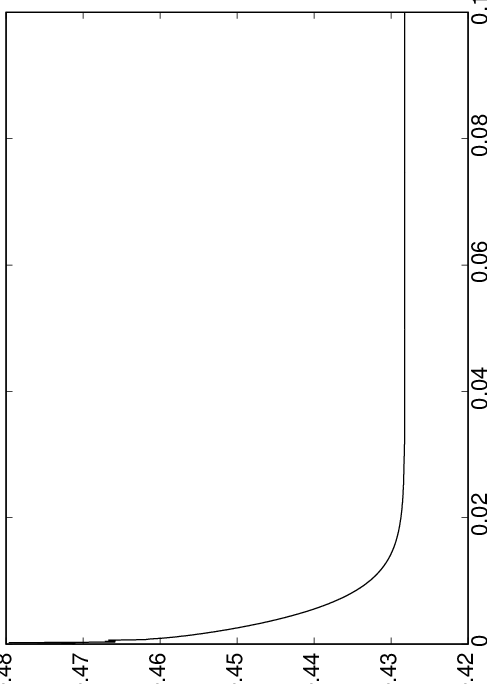}
\includegraphics[angle=-0,width=0.4\textwidth,align=t]{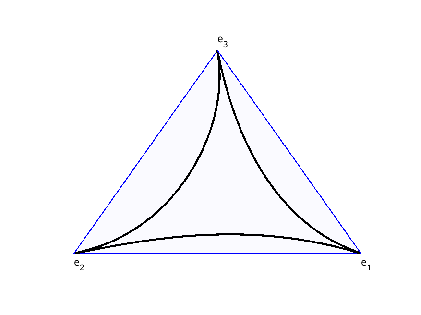}
\caption{$(\BGNmckappa_{m})^{h}$
Geodesic curvature flow for the metric \eqref{eq:gGNS} with \eqref{eq:Psi} 
and $(\sigma_{12},\sigma_{13},\sigma_{23},\sigma_{123}) = (4,6,9,0)$.
The solutions $\vec X^m$ at times $t = 0, 0.01, 0.1$. 
A plot of the discrete energy $L_g^h(\vec X^m)$ on the right.
Below a plot of the three minimisers connecting the vertices
of the Gibbs simplex.
} 
\label{fig:gGNSline}
\end{figure}%
In \cite{GarckeNS00} numerical computations indicated that on choosing 
$\sigma_{123} > 0$ in \eqref{eq:Psi} larger and larger, the minimising
profiles can be forced to approach the edges of the Gibbs simplex. 
To confirm this effect with our numerical method, we now choose 
$\sigma_{123} \in \{10,100,1000\}$ and plot the obtained geodesics in
Figure~\ref{fig:gGNS123}. It can be seen that for an increasing value of
$\sigma_{123}$, the geodesics are pushed further and further towards the edges
of the Gibbs simplex.
\begin{figure}
\centering
\mbox{
\includegraphics[angle=-0,width=0.33\textwidth,align=t]{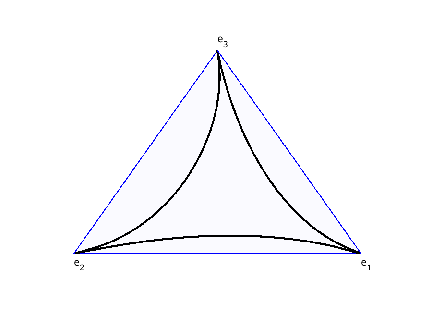}
\includegraphics[angle=-0,width=0.33\textwidth,align=t]{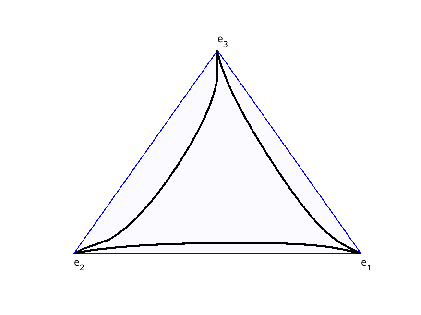}
\includegraphics[angle=-0,width=0.33\textwidth,align=t]{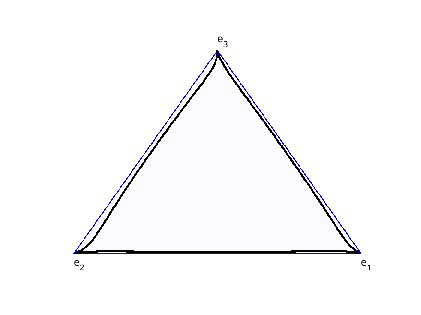}}
\caption{$(\BGNmckappa_{m})^{h}$
The minimisers obtained within the Gibbs simplex, 
for \eqref{eq:Psi} with $(\sigma_{12},\sigma_{13},\sigma_{23}) = (4,6,9)$
and $\sigma_{123} = 10,100,1000$ (from left to right).
} 
\label{fig:gGNS123}
\end{figure}%

We remark that in \cite{BretinM17} a novel approach for multi-component phase
field models has been considered, where the Ginzburg--Landau energy can be
defined such that the minimising paths connecting the pure phases are
always given by the edges of the Gibbs simplex. The metric that would
arise in the form of \eqref{eq:gGNS} in order to model this situation
is in general no longer conformal, and so would be outside the context of
this paper. However, following the approach in \cite{Haas07,GarckeH08}, we can
consider the following replacement of \eqref{eq:Psi} to achieve the same
effect:
\begin{equation} \label{eq:Psi2}
\Psi(u_1, u_2, u_3) = 
\sigma_{12}\,u_1^2\,u_2^2
+ \sigma_{13}\,u_1^2\,u_3^2 
+ \sigma_{23}\,u_2^2\,u_3^2 
+ \hat\sigma_{123}\,u_1\,u_2\,u_3^2 
+ \hat\sigma_{231}\,u_2\,u_3\,u_1^2 
+ \hat\sigma_{312}\,u_3\,u_1\,u_2^2 
\,,
\end{equation}
where $\hat\sigma_{123},\hat\sigma_{231},\hat\sigma_{312} \in \bRgeq$.
We perform a computation for \eqref{eq:Psi2} with
$\sigma_{12}=\sigma_{13}=\sigma_{23}=\hat\sigma_{123}=\hat\sigma_{231}=
\hat\sigma_{312}=1$ and show the obtained results in Figure~\ref{fig:gGH}.
It can be seen that now the geodesics lie on the edges of the Gibbs simplex,
confirming the analysis in \cite{Haas07,GarckeH08}.
\begin{figure}
\centering
\includegraphics[angle=-0,width=0.33\textwidth,align=t]{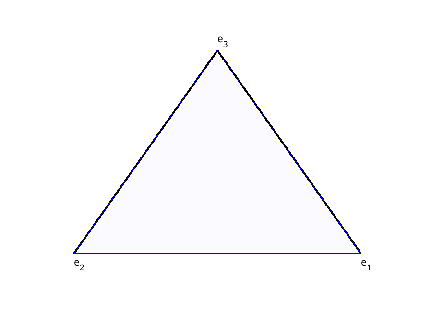}
\caption{$(\BGNmckappa_{m})^{h}$
The minimisers obtained for \eqref{eq:Psi2} with 
$\sigma_{12}=\sigma_{13}=\sigma_{23}=\hat\sigma_{123}=\hat\sigma_{231}=
\hat\sigma_{312}=1$ are precisely the edges of the Gibbs simplex.
} 
\label{fig:gGH}
\end{figure}%

\noindent
{\large\bf Acknowledgements}\\
The authors gratefully acknowledge the support 
of the Regensburger Universit\"atsstiftung Hans Vielberth.

\def\soft#1{\leavevmode\setbox0=\hbox{h}\dimen7=\ht0\advance \dimen7
  by-1ex\relax\if t#1\relax\rlap{\raise.6\dimen7
  \hbox{\kern.3ex\char'47}}#1\relax\else\if T#1\relax
  \rlap{\raise.5\dimen7\hbox{\kern1.3ex\char'47}}#1\relax \else\if
  d#1\relax\rlap{\raise.5\dimen7\hbox{\kern.9ex \char'47}}#1\relax\else\if
  D#1\relax\rlap{\raise.5\dimen7 \hbox{\kern1.4ex\char'47}}#1\relax\else\if
  l#1\relax \rlap{\raise.5\dimen7\hbox{\kern.4ex\char'47}}#1\relax \else\if
  L#1\relax\rlap{\raise.5\dimen7\hbox{\kern.7ex
  \char'47}}#1\relax\else\message{accent \string\soft \space #1 not
  defined!}#1\relax\fi\fi\fi\fi\fi\fi}

\end{document}